\documentclass[11pt, a4paper]{article}

\usepackage{amsmath}
\usepackage{amsfonts}
\usepackage{amssymb}
\usepackage[francais,english]{babel}
\usepackage{xcolor}

\def\english{\selectlanguage{english}}

\providecommand\mathbb{\bf}
\newcommand\R{{\mathbb R}}
\newcommand\N{{\mathbb N}}

\newtheorem{thm}{Theorem}[section]

\newtheorem{pro}{Proposition}[section]

\newtheorem{remark}{Remark}[section]

\newcounter{Remark}

\renewcommand\theRemark{\arabic{Remark}}

\newcounter{steps}
\newenvironment{proof}[1][]{%
\par\medbreak\setcounter{steps}{0}
{\noindent\bfseries Proof#1. }} {\hfill\fbox{\ }\medbreak}

\newcounter{substeps}[steps]

\newcommand{\olx}[0]{
\overline{x}}

\newcommand{\olv}[0]{
\overline{v}}

\newcommand{\polv}[0]{
{^\bot \overline{v}}}

\newcommand{\olX}[0]{
\overline{X}}

\newcommand{\olV}[0]{
\overline{V}}

\newcommand{\calr}[0]{
{\mathcal R}}

\newcommand{\cale}[0]{
{\mathcal E}}

\newcommand{\oc}[0]{
\omega_c}

\newcommand{\avepq}[2]{
\left \langle#1, #2\right \rangle_{P,Q}}

\newcommand{\intty}[1]{
\int _0 ^{+\infty}\!\!\!\!\int _{\R ^m} \!\!#1 \;\mathrm{d}y\mathrm{d}t}

\newcommand{\inttz}[1]{
\int _0 ^{+\infty}\!\!\!\!\int _{\R ^m} \!\!#1 \;\mathrm{d}z\mathrm{d}t}

\newcommand{\inttTy}[1]{
\int _0 ^{T}\!\!\!\!\int _{\R ^m} \!\!#1 \;\mathrm{d}y\mathrm{d}t}

\newcommand{\inty}[1]{
\int _{\R^m}\!\!#1 \;\mathrm{d}y}

\newcommand{\intz}[1]{
\int _{\R^m}\!\!#1 \;\mathrm{d}z}

\newcommand{\uin}[0]{
u^{\mathrm{in}}}

\newcommand{\dom}[0]{
\mathrm{dom}}

\newcommand{\cdny}[0]{
\cdot \nabla _y }

\newcommand{\cdnz}[0]{
\cdot \nabla _z }

\newcommand{\dyb}[0]{
\partial _y b}

\newcommand{\ysy}[0]{
Y(s;y)}

\newcommand{\supp}[0]{
\mathrm{supp\;}}

\newcommand{\eps}[0]{
\varepsilon}

\newcommand{\epsk}[0]{
{\varepsilon _k}}

\newcommand{\tepsk}[0]{
{\tilde{\varepsilon} _k}}

\newcommand{\supe}[0]{
\sup _{\varepsilon >0}}

\newcommand{\lime}[0]{
\lim _{\varepsilon \searrow 0 }}

\newcommand{\ue}[0]{
u ^\varepsilon}

\newcommand{\tue}[0]{
\tilde{u} ^\varepsilon}

\newcommand{\tueo}[0]{
\tilde{u} ^1_\varepsilon}

\newcommand{\ve}[0]{
v ^\varepsilon}

\newcommand{\fe}[0]{
f ^\varepsilon}

\newcommand{\Phie}[0]{
\Phi ^\varepsilon}

\newcommand{\psie}[0]{
\psi ^\varepsilon}

\newcommand{\md}[0]{
\mathrm{d}}

\newcommand{\vek}[0]{
v ^{\varepsilon _k}}

\newcommand{\limk}[0]{
\lim _{k \to +\infty}}

\newcommand{\divy}[0]{
\mathrm{div}_y}

\newcommand{\divz}[0]{
\mathrm{div}_z}

\newcommand{\loloc}[0]{
L^1_{\mathrm{loc}}(\R^m)}

\newcommand{\ltloc}[0]{
L^2_{\mathrm{loc}}(\R^m)}

\newcommand{\dpri}[0]{
{\cal D}^{\;\prime}(\R^m)}

\newcommand{\ran}[0]{
\mathrm{Range\;}}

\newcommand{\lty}[0]{
L^2(\R ^m)}

\newcommand{\liy}[0]{
L^\infty(\R ^m)}

\newcommand{\lilocy}[0]{
L^\infty_{\mathrm{loc}}(\R ^m)}

\newcommand{\bg}[0]{
b \cdot \nabla _y  }

\newcommand{\litlty}[0]{
{L^\infty (\R_+; L^2 (\R^m))}}

\newcommand{\lotlty}[0]{
{L^1 (\R_+; L^2 (\R^m))}}

\newcommand{\ave}[1]{
\left \langle #1 \right \rangle }

\newcommand{\ytz}[0]{
Y\!\left ( t/\varepsilon; z \right )}

\newcommand{\yt}[0]{
Y\!\left ( t/\varepsilon; \cdot \right )}
\newcommand{\ysz}[0]{
Y\!\left ( s; z \right )}

\newcommand{\ymty}[0]{
Y\left ( -t/\varepsilon; y\right )}

\newcommand{\ymsy}[0]{
Y\left ( -s; y\right )}

\newcommand{\ymsysz}[0]{
Y\left ( - s;Y\left ( s;z\right) \right)}

\newcommand{\dymtz}[0]{
\partial Y ^{-1} (t/\varepsilon; z)}

\newcommand{\dymsz}[0]{
\partial Y^{-1}(s;z)}

\newcommand{\limS}[0]{
\lim _{S \to +\infty}}

\newcommand{\ymsys}[0]{
Y(-s; Y(s; \cdot))}

\newcommand{\ys}[0]{
Y(s; \cdot)}

\newcommand{\ymt}[0]{
Y(-t/\varepsilon; \cdot)}

\newcommand{\proj}[1]{
\mathrm{Proj}_{#1}}

\newcommand{\litloclty}[0]{
L^\infty_{\mathrm{loc}}(\R_+;L^2(\R^m))}

\newcommand{\litlocltz}[0]{
L^\infty_{\mathrm{loc}}(\R_+;L^2(\R^m))}

\newcommand{\nry}[0]{
\,\nabla _y ^R}

\newcommand{\nrz}[0]{
\,\nabla _z ^R}

\newcommand{\lttlocxp}[0]{
L^2_{\mathrm{loc}}(\R_+;X_P)}

\newcommand{\lttlocltz}[0]{
L^2_{\mathrm{loc}}(\R_+;L^2(\R^m))}

\newcommand{\lotlocltz}[0]{
L^1_{\mathrm{loc}}(\R_+;L^2(\R^m))}

\newcommand{\calO}[0]{
\mathcal{O}}

\newcommand{\calB}[0]{
\mathcal{B}}

\newcommand{\calOs}[0]{
\mathcal{O}(s; \cdot)}

\newcommand{\dys}[0]{
\partial Y(s; \cdot)}

\newcommand{\dyms}[0]{
\partial Y ^{-1} (s; \cdot)}

\newcommand{\intS}[1]{
\int_0 ^S #1 \;\mathrm{d}s}

\newcommand{\lttxp}[0]{
L^2 (\R_+;X_P)}

\newcommand{\Tde}[0]{
{T_{\delta, \eps}}}

\newcommand{\kde}[0]{
{k_{\delta, \eps}}}

\newcommand{\inttT}[1]{
\int_0^T #1 \;\mathrm{d}t}

\begin{document}
\english

\title{Asymptotic analysis of parabolic equations with stiff transport terms by a multi-scale approach}

\author{Thomas BLANC \thanks{Aix Marseille Universit\'e, CNRS, Centrale Marseille, Institut de Math\'ematiques de Marseille, UMR 7373, Ch\^ateau Gombert 39 rue F. Joliot Curie, 13453 Marseille 
FRANCE. E-mail : {\tt thomas.blanc@univ-amu.fr}}\;,\; Miha\"i BOSTAN \thanks{Aix Marseille Universit\'e, CNRS, Centrale Marseille, Institut de Math\'ematiques de Marseille, UMR 7373, Ch\^ateau Gombert 39 rue F. Joliot Curie, 13453 Marseille FRANCE. E-mail : {\tt mihai.bostan@univ-amu.fr}}\;,\; Franck BOYER
\thanks{Universit\'e Toulouse 3 - Paul Sabatier, CNRS, Institut de Math\'ematiques de Toulouse, UMR 5219, Route de Narbonne 31062 Toulouse Cedex 9
FRANCE. E-mail : {\tt franck.boyer@math.univ-toulouse.fr}}
}

\date{ (\today)}

\maketitle

\begin{abstract}
We perform the asymptotic analysis of parabolic equations with stiff transport terms. This kind of problem occurs, for example, in collisional gyrokinetic theory for tokamak plasmas, where the velocity diffusion of the collision mechanism is dominated by the velocity advection along the Laplace force corresponding to a strong magnetic field. This work appeal to the filtering techniques. Removing the fast oscillations associated to the singular transport operator, leads to a stable family of profiles. The limit profile comes by averaging with respect to the fast time variable, and still satisfies a parabolic model, whose diffusion matrix is completely characterized in terms of the original diffusion matrix and the stiff transport operator. Introducing first order correctors allows us to obtain strong convergence results, for general initial conditions (not necessarily well prepared).
\end{abstract}



\section{Introduction}
\label{intro}

In many applications we deal with disparate scales. The solutions of the problems in hand fluctuate at very different scales and for the moment, solving numerically for both slow and fast scales seems out of reach. Depending on the particular regimes we are interested on, it could be worth to solve with respect to the slow variable, after smoothing out the fast oscillations. In this work we focus on parabolic models perturbed by stiff transport operators

\begin{equation}
\label{Equ1}
\left\{
\begin{array}{ll}
\partial _t \ue - \divy (D(y) \nabla _y \ue ) + \displaystyle\frac{1}{\eps}\; b(y) \cdot \nabla _y \ue = 0,& (t,y) \in \R_+ \times \R^m\\
\ue (0,y) = \uin (y), & y \in \R^m.
\end{array}
\right.
\end{equation}
Here $b : \R^m \to \R^m$ and $D : \R ^m \to \mathcal{M}_m (\R)$ are given fields of vectors and symmetric matrices, and $\eps >0$ is a small parameter destinated to converge to $0$. If the vector field $b$ is divergence free, the energy balance writes
\[
\frac{1}{2}\frac{\md}{\md t} \inty{(\ue (t,y))^2} + \inty{D(y) \nabla _y \ue \cdot \nabla _y \ue } = 0.
\]
Therefore, when the matrices $D(y)$ are positive, the $L^2$ norms of the solutions $(\ue)_{\eps >0}$ decrease in time, and we expect that the limit model still behaves like a parabolic one, whose diffusion matrix field is to be determined. This work is motivated by the study of collisional models for the gyrokinetic theory in tokamak plasmas. The fluctuations of the presence density of charged particles are due to the transport in space and velocity (under the action of  electro-magnetic fields), but also to the collision mechanisms. In the framework of the magnetic confinement fusion, the external magnetic fields are very large, leading to a stiff velocity advection, due to the magnetic force $q v \wedge B^\eps = q v \wedge \frac{B}{\eps}$. Here $q$ stands for the particle charge and $B^\eps = \frac{B}{\eps}$ represents a strong magnetic field, when $\eps$ goes to $0$. Using a Fokker-Planck operator for taking into account the collisions between particles, we are led to the Fokker-Planck equation
\begin{equation}
\label{EquFP}
\partial _t \fe + v \cdot \nabla _x \fe + \frac{q}{m} \left ( E + v \wedge \frac{B}{\eps} \right ) \cdot \nabla _v \fe = \nu \;\mathrm{div}_v ( \Theta \nabla _v \fe + v \fe  ),\;\;(t,x,v) \in \R_+ \times \R^3 \times \R^3
\end{equation}
where $E$ is the electric field, $m$ is the particle mass, $\nu$ is the collision frequency and $\Theta$ is the temperature. The asymptotic analysis of \eqref{EquFP}, when neglecting the collisions is now well understood \cite{BosTraEquSin, FreSon98, FreSon01, GolSai99}. It can be handled by averaging the perturbed model along the characteristic flow associated to the dominant transport operator. Recently, models including collisions have been analyzed formally by using the averaging method \cite{BosCalQAM1, BosCalQAM2}. In particular, it was emphasized that, averaging with respect to the fast cyclotronic motion leads to diffusion not only in velocity, but also with respect to the perpendicular space directions, see \eqref{EquFPDiffusion}. The study of the averaged diffusion matrix field is crucial when determining the equilibria of the limit Fokker-Planck equation \eqref{EquFP}, when $\eps$ goes to zero. Numerical results concerning strongly anisotropic elliptic and parabolic problems were obtained in \cite{DegDelNeg09, FilNegYan12, CroKuhLat15}.

This work concentrates on the asymptotic analysis for the parabolic models in \eqref{Equ1}. We expect that part of these arguments applies to other perturbed models, for example in the framework of strongly anisotropic parabolic models, which will be studied in future works. Our paper is organized as follows. The main results are introduced in Section \ref{MainRes}. We indicate the main lines of our arguments, performing formal computations. In Section \ref{AveOpe} we present a brief overview on the construction of the average operators for matrix fields. Section \ref{UnifEst} is devoted to uniform estimates, in view of convergence results. In Section \ref{2scale} we establish two-scale convergence results, in the ergodic setting, which allows us to handle situations with non periodic fast variables. Up to our knowledge, these results have not been reported yet. The proofs of the main theorems are detailed in Section \ref{Proofs}. Some technical arguments are presented in Appendix \ref{A}.

\section{Presentation of the main results and formal approach}
\label{MainRes}

The subject matter of this paper concentrates on the asymptotic analysis of \eqref{Equ1}, when $\eps$ becomes small. Obviously, the fast time oscillations come through the large advection field $\frac{b(y)}{\eps} \cdot \nabla _y$. Indeed, think that when neglecting the diffusion operator, the problem \eqref{Equ1} reduces to a transport model, whose solution writes
\begin{equation}
\label{Equ3}
\ue (t,y) = \uin (\ymty ),\;\;(t, y) \in \R_+ \times \R^m.
\end{equation}
Here $(s,y) \in \R \times \R^m \to Y(s;y) \in \R^m$ stands for the characteristic flow of $\bg$
\begin{equation*}
\label{Equ4} \frac{\md Y}{\md s} = b(\ysy),\;\;(s,y) \in \R \times \R^m,\;\;Y(0;y) = y,\;\;y \in \R^m.
\end{equation*}
This flow is well defined under standard smoothness assumptions
\begin{equation}
\label{Equ15} b \in W^{1,\infty}_{\mathrm{loc}} (\R^m),\;\;\divy b = 0
\end{equation}              
and
\begin{equation}
\label{Equ17}
\exists C >0\;\mbox{such that } |b(y) | \leq C ( 1 + |y|),\;\;y \in \R^m.
\end{equation}    
Under the above hypotheses the flow $Y$ is global and smooth, $Y\in W^{1,\infty}_{\mathrm{loc}}(\R \times \R^m)$. Moreover, since the field $b$ is divergence free, the transformation $y \in \R^m \to \ysy \in \R^m$ is measure preserving for any $s \in \R$. Motivated by \eqref{Equ3}, we introduce the new unknowns
\begin{equation}
\label{EquNewUnknown}
\ve (t,z) = \ue (t, \ytz),\;\;(t,z) \in \R_+ \times \R^m,\;\;\eps >0
\end{equation}
and we expect to get stability for the family $(\ve)_{\eps >0}$, when $\eps$ goes to $0$. In that case we will deduce that, for small $\eps >0$, $\ue$ behaves like $v(t, \ymty)$, for some profile $v = \lime \ve$, that is, $\ue$ appears as the composition product between a stable profile and the fast oscillating flow $\ymty$. We prove mainly two strong convergence results for general initial conditions (not necessarily well prepared), whose simplified versions are stated below. For detailed assertions see Theorems \ref{MainRes1}, \ref{MainRes2}.

\medskip
\noindent
{\bf Theorem}
{\it 
We denote by $(\ue)_{\eps>0}$ the variational solutions of \eqref{Equ1} and by $(\ve)_{\eps>0}$ the functions
\[
\ve (t,z) = \ue (t, \ytz),\;\;(t,z) \in \R_+ \times \R^m,\;\;\eps >0.
\]
\begin{enumerate}
\item
Under suitable hypotheses on the vector field $b$, the matrix field $D$ and the initial condition $\uin$, the family $(\ve)_{\eps>0}$ converges strongly in $\litloclty$ to the unique variational solution $v \in \litlty$ of \eqref{Equ14}, whose diffusion matrix field $\ave{D}$ comes by averaging the matrix field $D$ along the flow of the vector field $b$ (cf. Theorem \ref{AveMatField}). 

\item
Under more regularity hypotheses, we have 
\[
\ue (t, \cdot) = v(t, \ymt) + \mathcal{O}(\eps) \mbox{ in } \litloclty
\]
that is, for any $T \in \R_+$, there is a constant $C_T$ such that 
\[
\sup _{t \in [0,T]}\|\ue (t, \cdot) - v(t, \ymt)\|_{\lty} \leq C_T \eps.
\]
\end{enumerate}
}
\noindent
The problem satisfied by $\ve$ is obtained by performing the change of variable $y = \ytz$ in \eqref{Equ1}. A straightforward  computation based on the chain rule leads to (see Remark \ref{ChainRuleDiff}) 
\[
\partial _t \ve (t,z) = \partial _t \ue (t, \ytz) + \frac{1}{\eps}\; b(\ytz) \cdot (\nabla _y \ue)(t, \ytz)
\]
and
\[
\divz \{ \dymtz D (\ytz) \;^t \dymtz \nabla _z \ve \} = \{\divy ( D(y) \nabla _y \ue ) \}(t, \ytz)
\]
and therefore \eqref{Equ1} becomes
\begin{equation}
\label{Equ8}
\left \{
\begin{array}{lll}
\partial _t \ve - \divz \left ( (G(t/\eps) D )\nabla _z \ve \right ) = 0,& (t,z) \in \R_+ \times \R^m\\
\ve (0,z) = \ue (0,z) = \uin (z), & z \in \R^m, \;\eps >0
\end{array}
\right.
\end{equation}
where $(G(s)D)_{s \in \R}$ is the family of matrix fields given by
\begin{align}
\label{EquG}
(G(s)D)(z) & = \dymsz D(\ysz) \;^t \dymsz \\
& = \partial \ymsysz D(\ysz) \;^t \partial \ymsysz,\;\;(s, z) \in \R \times \R^m.\nonumber 
\end{align}
The new diffusion problem \eqref{Equ8} seems simpler than the original problem \eqref{Equ1}, because the singular term $\frac{1}{\eps}\; \bg$ has disappeared. Nevertheless, the new model depends on a fast time variable $s = t/\eps$, through the diffusion matrix field $G(s= t/\eps)D$, and a slow time variable $t$. We deal with a two-scale problem in time. As often in asymptotic analysis of multiple scale problems, a way to understand the behavior of the solutions $(\ve)_{\eps >0}$ when $\eps$ goes to $0$ and to identify the limit problem is to use a formal development whose terms depend both on the slow and fast time variables
\begin{equation}
\label{Equ10}
\ve (t,z) = v(t,  t/\eps, z) + \eps v^1 (t,  t/\eps, z) + ... .
\end{equation}
This method is used in many frameworks such as periodic homogenization for elliptic and parabolic systems \cite{Allaire, Guetseng}, transport equations \cite{Bos15, Dal08} or kinetic equations \cite{Bos12}. Plugging the Ansatz \eqref{Equ10} in \eqref{Equ8} and identifying the terms of the same order with respect to $\eps$, lead to the hierarchy of equations
\begin{equation}
\label{Equ11}
\partial _s v = 0
\end{equation}
\begin{equation}
\label{Equ12}
\partial _t v - \divz (G(s) D \nabla _z v ) + \partial _s v^1 = 0
\end{equation}
\[
\vdots
\]
Equation \eqref{Equ11} says that the first profile $v$ does not depend on the fast time variable $s$, that is $v = v(t,z)$. We expect that $v$ is the limit of the family $(\ve)_{\eps >0}$, when $\eps$ goes to $0$. The slow time evolution of $v$ is given by \eqref{Equ12}, but we need to eliminate the second profile $v^1$. Actually $v^1$ appears as a Lagrange multiplier which guarantees that at any time $t$, the profile $v$ satisfies the constraint $\partial _s v = 0$. In the periodic case, we eliminate $v^1$ by taking the average over one period. More general, we appeal to ergodic average and we write
\begin{equation}
\label{EquLimMod}
\left \{
\begin{array}{lll}
\partial _t v - \divz\left \{ \left ( \displaystyle \lim _{S\to +\infty} \frac{1}{S} \int _0 ^S G(s) D \;\md s \right )\nabla _z v\right \} = 0,& \;\;(t,z) \in \R_+ \times \R^m \\
v(0,z) = \uin (z), & \;\; z \in \R^m.
\end{array}
\right.
\end{equation}
The key point is that $(G(s))_{s\in \R}$ is a $C^0$-group of unitary operators (on some Hilbert space to be determined), and thanks to von Neumann's ergodic mean theorem \cite{ReedSimon}, the limit $\ave{D} = \lim _{S\to +\infty} \frac{1}{S} \int _0 ^S G(s) D \;\md s$ makes sense. The Hilbert space which realizes $(G(s))_{s\in \R}$ as a $C^0$-group of unitary operators appears as a $L^2$ weighted space, with respect to some field of symmetric definite positive matrices. We assume that there is a matrix field $P$ such that 
\begin{equation}
\label{Equ32Bis}
^t P = P,\;\;P(y)\xi \cdot \xi >0,\;\;\xi \in \R^m \setminus \{0\},\;\;y \in \R^m,\;\;P^{-1}, P \in L^2_{\mathrm{loc}}(\R^m)
\end{equation}
\begin{equation}
\label{Equ33} [b,P] := (b \cdot \nabla _y )P - \partial _y b P - P {^t \partial _y b} = 0,\;\;\mbox{ in } \dpri{}.
\end{equation}
For example, when the vector field $b$ is uniform, we can take $P = I_m$. Notice that we have the following characterization for \eqref{Equ33} cf. Proposition 3.8 \cite{BosAnisoDiff}
\begin{pro}
\label{CaracBraMat} 
Consider $b \in W^{1,\infty}_{\mathrm{loc}} (\R^m)$ (not necessarily divergence free) with at most linear growth at infinity and $A(y) \in \loloc{}$. Then $[b,A] = 0$ in $\dpri{}$ iff
\begin{equation*}
\label{Equ34} A(\ysy) = \partial \ysy A(y) \;{^t \partial \ysy},\;\;s\in \R,\;\;y \in \R^m.
\end{equation*}
\end{pro}
Given a matrix field $P$ satisfying \eqref{Equ32Bis}, \eqref{Equ33}, we consider the set of matrix fields
\[
H_{Q}=\left\{ A:\R^m \to \mathcal{M}_m(\R)\;\text{measurable}\;: \;Q^{1/2}AQ^{1/2} \in L^2 \right\}
\]
where $Q=P^{-1}$, and the scalar product on $H_{Q}$
\[
(A,\,B)_{Q}=\inty{Q^{1/2}(y)\,A(y)\,Q^{1/2}(y) : Q^{1/2}(y)\,B(y)\,Q^{1/2}(y)} = \inty{QA:BQ}
\]
for any $A, B \in H_Q$. For any two matrices in $\mathcal{M}_m (\R)$, the notation $A:B$ stands for $\mathrm{tr}(^tAB)$. 
Notice that the application $J : H_Q \to L^2(\R^m; \mathcal{M}_m(\R))$, given by
$J(A) = Q^{1/2}\,A\,Q^{1/2}, A \in H_Q$ is an isometry, implying that $(H_Q, (\cdot, \cdot))$ is a Hilbert space. We prove that the family of applications $G(s) : H_Q \to H_Q, s \in \R$, is a $C^0$-group of unitary operators on $H_Q$ cf. Proposition \ref{CZGroup}. Thanks to Theorem \ref{vonNeumann} (see \cite{ReedSimon} for more details), the average of a matrix field $\ave{A} := \lim _{S\to +\infty} \frac{1}{S} \int _0 ^S G(s)A\;\md s$ is well defined and coincides with the orthogonal projection on $\{B \in H_Q\;:\;G(s)B = B\;\mbox{for any } s\in \R\}$. 
\begin{thm}
\label{AveMatField}
Assume that \eqref{Equ15}, \eqref{Equ17}, \eqref{Equ32Bis}, \eqref{Equ33} hold true. We denote by $L$ the infinitesimal generator of the group $(G(s))_{s\in \R}$.
\begin{enumerate}
\item
For any matrix field $A \in H_Q$ we have the strong convergence in $H_Q$
\[
\ave{A} :=\limS \frac{1}{S} \int_r ^{r+S} \partial  \ymsys A(\ys) \;{^t \partial } \ymsys\;\md s = \proj{\ker L}A
\]
uniformly with respect to $r \in \R$. 
\item
If $A\in H_Q$ is a field of symmetric positive matrices, then so is $\ave{A}$. 
\item
If $A \in H_Q$ and there is $\alpha >0$ such that 
\[
Q^{1/2} (y) A(y) Q^{1/2}(y)\geq \alpha I_m,\;\;y \in \R^m
\]
therefore we have
\[
Q^{1/2} (y) \ave{A}(y) Q^{1/2}(y)\geq \alpha I_m,\;\;y \in \R^m
\]
and in particular, $\ave{A}(y)$ is definite positive for $y \in \R^m$. 
\item
If $A \in H_Q \cap H_Q ^\infty$ (see \eqref{HQinf} for the definition of the Banach space $H_Q^\infty$), then $\ave{A} \in H_Q \cap H_Q ^\infty$ and
\[
|\ave{A} |_Q \leq |A|_Q,\;\;|\ave{A}|_{H_Q ^\infty} \leq |A|_{H_Q^\infty}.
\]
\end{enumerate}
\end{thm}
In view of Theorem \ref{AveMatField}, the limit of the parabolic problems \eqref{Equ8} becomes, accordingly to \eqref{EquLimMod}
\begin{equation}
\label{Equ14}
\left\{
\begin{array}{ll}
\partial_t v- \divz ( \ave{D} \nabla _z v) =0, & (t,z)\in \R_+ \times \R^m\\
v(0,z) = \uin (z), & z \in \R^m.
\end{array}
\right.
\end{equation}
Under some regularity assumptions (see Section \ref{Proofs}), we obtain a strong convergence result for the family $(\ve)_{\eps>0}$ in $\litloclty$, toward the solution $v$ of the problem \eqref{Equ14}. Coming back to the family $(\ue)_{\eps >0}$, through the variable change in \eqref{EquNewUnknown}, and thanks to the fact that for any $s \in \R$, $\ys$ is measure preserving, we justify that at any time $t \in \R_+, \eps >0$, $\ue (t, \cdot)$ behaves (in $\lty$) like the composition product between $v(t, \cdot)$ and $\ymt$, that is 
\[
\lime \inty{( \;\ue (t,y) - v ( t, \ymty)\;)^2 } = \lime \intz{(\;\ve (t,z) - v(t,z)\;)^2 } = 0
\]
uniformly with respect to $t \in [0,T]$, for any $T \in \R_+$. 
\begin{thm}
\label{MainRes1}
Assume that the hypotheses \eqref{Equ15}, \eqref{Equ17}, \eqref{Equ58}, \eqref{Equ58Bis}, \eqref{E45}, \eqref{E46}, \eqref{E49} hold true together with all the regularity conditions in Proposition \ref{MoreEstimatesBis}. We suppose that $\uin \in H^2_R$ (see \eqref{HoneR}, \eqref{HtwoR} for the definitions of $H^k_R$ and $\nabla ^R$) and we denote by $(\ue)_{\eps>0}$ the variational solutions of \eqref{Equ1} and by $(\ve)_{\eps>0}$ the functions
\[
\ve (t,z) = \ue (t, \ytz),\;\;(t,z) \in \R_+ \times \R^m,\;\;\eps >0.
\]
Then the family $(\ve)_{\eps>0}$ converges strongly in $\litloclty$ to the unique variational solution $v \in \litlty$ of \eqref{Equ14}. The function $v$ has the regularity 
\[
\partial _t v, \;\nrz v, \;\nrz \otimes \nrz v   \in \litlocltz, \;\;\partial _t \nrz v \in \lttlocltz
\]
and $(\nabla _z \ve )_{\eps >0}$ converges toward $\nabla _z v$ in $L^2_{\mathrm{loc}}(\R_+; X_P)$ when $\eps$ goes to $0$ (see Section \ref{AveOpe} for the definition of the Hilbert space $X_P$). 
\end{thm}
Under additional hypotheses we can justify that $\ve = v + \mathcal{O}(\eps)$ in $\litlocltz$, as suggested by the formal Ansatz \eqref{Equ10}.
\begin{thm}
\label{MainRes2}
Assume that the hypotheses \eqref{Equ15}, \eqref{Equ17}, \eqref{Equ58}, \eqref{Equ58Bis}, \eqref{E45}, \eqref{E46} hold true. Moreover, we assume that the solution $v$ of the limit model \eqref{Equ14} is smooth enough, that is
\[
\nrz v \in \litlocltz, \;\;\nrz \otimes \nrz v \in \litlocltz
\]
\[
\nrz \partial _t v \in \lotlocltz, \;\;\nrz \otimes \nrz \partial _t v \in \lotlocltz
\]
\[
\nrz \otimes \nrz \otimes \nrz v \in \lttlocltz,\;\;\nrz \otimes \nrz \otimes \nrz \otimes \nrz v \in \lotlocltz
\]
and that there is a smooth matrix field $C$, that is
\[
\divy(RC), \;\;b_k \cdny \divy(RC), \;\;b_l \cdny (b_k \cdny \divy(RC)) \in \liy,\;\;k,l \in \{1, ..., m\}
\]
\[
RC\;^tR,\;\;b_k \cdny (RC\;^tR),\;\;b_l \cdny (b_k \cdny (RC\;^tR))\in \liy,\;\;k,l \in \{1, ..., m\}
\]
such that the following decomposition holds true
\[
D = \ave{D} + L(C),\;\;C \in (\ker L )^\bot.
\]
We denote by $(\ue)_{\eps>0}$ the variational solutions of \eqref{Equ1}. Then for any $T\in \R_+$, there is a constant $C_T$ such that 
\[
\sup _{t \in [0,T]} \|\ue (t, \cdot) - v(t, \ymt)\|_{\lty} \leq C_T \eps
\]
\[
\left ( \int _0 ^T |\nabla _y \ue (t, \cdot) - \nabla _y v(t, \ymt)|_P ^2 \;\md t\right ) ^{1/2}\leq C_T \eps.
\]
\end{thm}

\section{The average of a matrix field}
\label{AveOpe}
Consider a matrix field $P$ satisfying the hypotheses \eqref{Equ32Bis}, \eqref{Equ33} and the inverse matrix field $Q = P^{-1}$. We introduce the set
\[
H_Q = \left \{A:\R^m \to \mathcal{M}_m (\R)\;:\; \inty{Q(y)A(y) : A(y) Q(y)} < +\infty \right \}
\]
and the application 
\[
(\cdot, \cdot)_Q : H_Q \times H_Q \to \R,\;\;(A,B)_Q = \inty{Q(y)A(y):B(y)Q(y)}, \;A, B \in H_Q.
\]
It is easily seen that the bilinear application $(\cdot, \cdot)_Q$ is symmetric and positive definite and that the set $H_Q$ endowed with the scalar product $(\cdot,\cdot)_Q$ is a Hilbert space, whose norm is denoted by $|A|_Q = (A, A)_Q ^{1/2}, A \in H_Q$. Clearly $C^0_c(\R^m,\mathcal{M}_m(\R)) \subset H_Q$. 
Observe that $H_Q \subset L^1_{\mathrm{loc}}(\R^m;\mathcal{M}_m(\R))$. Indeed, if for any matrix $M$ the notation $|M|$ stands for the norm subordonated to the euclidian norm of $\R^m$
\[
|M| = \sup _{\xi \in \R^m \setminus \{0\}} \frac{|M\xi|}{|\xi|} \leq ( M : M ) ^{1/2}
\]
we have
\begin{eqnarray*}
\label{Equ57} |A| & = & \sup _{\xi, \eta \neq 0} \displaystyle \frac{A \xi \cdot \eta}{|\xi|\;|\eta|} \\
& = & \sup _{\xi, \eta \neq 0} \displaystyle \frac{Q^{1/2}AQ^{1/2} P ^{1/2}\xi \cdot P^{1/2} \eta}{|P^{1/2} \xi|\;|P^{1/2} \eta|}\;\frac{|P^{1/2}\xi|}{|\xi|}\;\frac{|P^{1/2}\eta|}{|\eta|} \nonumber \\
& \leq & |Q^{1/2}AQ^{1/2} |\;|P^{1/2}|^2 \nonumber \\
& \leq & ( Q^{1/2}AQ^{1/2}:Q^{1/2}AQ^{1/2}) ^{1/2} \;|P|.\nonumber 
\end{eqnarray*}
We deduce that for any $R>0$
\[
\int_{B_R} |A(y)|\;\md y \leq \int _{B_R} ( Q^{1/2}AQ^{1/2}:Q^{1/2}AQ^{1/2}) ^{1/2} \;|P|\;\md y \leq (A,A)_Q ^{1/2} \left ( \int _{B_R} |P(y)|^2 \;\md y \right ) ^{1/2}.
\]
When replacing the matrix field $Q$ by the matrix field $P$, we obtain the Hilbert space
\[
H_P = \left \{A:\R^m \to \mathcal{M}_m(\R)\;:\; \inty{P(y)A(y) : A(y) P(y)} < +\infty \right \}
\]
endowed with the scalar product 
\[
(\cdot, \cdot)_P : H_P \times H_P \to \R,\;\;(A,B)_P = \inty{P(y)A(y):B(y)P(y)}, \;A, B \in H_P.
\]
Motivated by the computations leading to \eqref{EquG}, we consider the family of linear transformations $(G(s))_{s\in\R}$, acting on matrix fields. It happens that $(G(s))_{s\in\R}$ is a $C^0$-group of unitary operators on $H_Q$ (see \cite{BosAnisoDiff} Proposition 3.12 for details). For any function $f = f(y), y \in \R^m$, the notation $f_s = f_s(z)$ stands for the composition product $f_s = f \circ \ys$. 
\begin{pro}
\label{CZGroup} 
Assume that the hypotheses \eqref{Equ15}, \eqref{Equ17}, \eqref{Equ32Bis}, \eqref{Equ33} hold true.
\begin{enumerate}
\item
The family of applications 
\[
A \to G(s)A : = \partial  Y ^{-1} (s; \cdot) A_s \; {^t \partial  Y } ^{-1} (s; \cdot) = \partial \ymsys A_s \;^t \partial \ymsys
\]
is a $C^0$-group of unitary operators on $H_Q$.
\item
If $A$ is a field of symmetric matrices, then so is $G(s)A$, for any $s \in \R$.
\item
If $A$ is a field of positive matrices, then so is $G(s)A$, for any $s \in \R$. 
\item
If there is $\alpha >0$ such that $Q^{1/2}(y)A(y)Q^{1/2}(y) \geq \alpha I_m, y \in \R^m$, then for any $s\in \R$ we have $Q^{1/2}(y)(G(s)A)(y)Q^{1/2}(y) \geq \alpha I_m, y \in \R ^m$. 
\end{enumerate}
\end{pro}
\begin{proof}$\;$\\
1. Thanks to the characterization  in Proposition \ref{CaracBraMat} we know that
\begin{equation}
\label{EquCaracP}
P_s = \partial \ys P \;^t \dys,\;\;s\in \R.
\end{equation}
For any $s \in \R$ we consider the matrix field $\calOs = Q_s ^{1/2} \dys  Q^{-1/2}$. Observe that $\calOs$ is a field of orthogonal matrices, for any $s \in \R$. Indeed we have, thanks to \eqref{EquCaracP}
\begin{align*}
^t \calOs \calOs & = Q^{-1/2} \;^t \dys Q_s ^{1/2} Q_s ^{1/2} \dys Q^{-1/2}\\
& = Q^{-1/2} \left ( \dyms P_s \;^t \dyms\right )^{-1} Q^{-1/2} \\
& = Q^{-1/2} P^{-1} Q^{-1/2} \\
& = I_m
\end{align*}
implying that for any matrix field $A$ we have
\begin{align}
\label{EquQGs}
Q^{1/2} G(s)A Q^{1/2} & = Q^{1/2} \dyms A_s \;^t \dyms Q^{1/2} = \;^t \calOs Q_s ^{1/2} A_s Q_s ^{1/2} \calOs.
\end{align}
It is easily seen that if $A \in H_Q$, then for any $s \in \R$
\begin{align*}
|G(s)A|^2_Q & = \inty{Q^{1/2} G(s)A Q^{1/2} : Q^{1/2} G(s)A Q^{1/2}}\\
& = \inty{\;^t \calOs Q_s ^{1/2} A_s Q_s ^{1/2} \calOs : \;^t \calOs Q_s ^{1/2} A_s Q_s ^{1/2} \calOs} \\
& = \inty{Q_s ^{1/2} A_s Q_s ^{1/2} : Q_s ^{1/2} A_s Q_s ^{1/2}}\\
& = \inty{Q^{1/2} A Q^{1/2} :Q^{1/2} A Q^{1/2} } = |A|^2 _Q
\end{align*}
saying that $G(s)$ is a unitary transformation for any $s\in \R$. The group property of the family $(G(s))_{s\in \R}$ follows easily from the group property of the flow $(\ys)_{s \in \R}$
\begin{align*}
G(s) G(t) A & = \dyms (G(t)A)_s {^t \dyms}\\
& = \dyms \partial  Y ^{-1} (t; Y(s;\cdot))(A_t)_s \,{^t \partial  Y ^{-1} (t; Y(s;\cdot))}\,{^t \dyms} \\
& = \partial  Y ^{-1} (t + s;\cdot)A_{t+s} \,{^t \partial  Y ^{-1} (t + s;\cdot)} = G(t+s) A,\;\;A \in H_Q.
\end{align*}
The continuity of the group, {\it i.e.,} $\lim _{s\to 0} G(s)A = A$ strongly in $H_Q$, is left to the reader.\\
2. Notice that $G(s)$ commutes with transposition
\begin{align*}
^t (G(s)A) & = \;^t \left ( \dyms A_s \;^t \dyms  \right )\\
& = \dyms \;^t A_s \;^t \dyms \\
& = G(s) \;^t A.
\end{align*}
In particular, if $^tA = A$, then $^t (G(s)A) = G(s)A$.\\
3. We use the formula \eqref{EquQGs}. For any $\xi, \eta \in \R^m$, the notation $\xi \otimes \eta$ stands for the matrix whose $(i,j)$ entry is $\xi _i \eta_j$. For any $\xi \in \R^m$ we have
\begin{align*}
G(s)A : Q^{1/2} \xi \otimes Q^{1/2}\xi & = Q^{1/2}G(s)A Q^{1/2} : \xi \otimes \xi \\
& = \;^t \calOs Q_s ^{1/2} A_s Q_s ^{1/2} \calOs : \xi \otimes \xi \\
& = Q_s ^{1/2} A_s Q_s ^{1/2} : \calOs (\xi \otimes \xi ) \;^t \calOs \\
& = Q_s ^{1/2} A_s Q_s ^{1/2} : (\calOs \xi ) \otimes (\calOs \xi)\\
& = A_s : (Q_s ^{1/2} \calOs \xi) \otimes ( Q_s ^{1/2} \calOs \xi ).
\end{align*}
As $A$ is a field of positive matrices, therefore $G(s)A$ is a field of positive matrices as well. \\
4. Assume that there is $\alpha >0$ such that $Q^{1/2} A Q^{1/2}\geq \alpha I_m$. As before we write for any $\xi \in \R^m$
\[
Q^{1/2} G(s)A Q^{1/2} : \xi \otimes \xi  = (Q^{1/2}AQ^{1/2})_s : ( \calOs \xi ) \otimes (\calOs \xi ) \geq \alpha |\calOs \xi |^2 = \alpha |\xi |^2
\]
saying that $Q^{1/2} G(s)A Q^{1/2} \geq \alpha I_m$. 
\end{proof}
We denote by $L$ the infinitesimal generator of the group $G$
\[
L:\dom(L) \subset H_Q \to H_Q,\;\;\dom L = \{ A\in H_Q\;:\; \exists \;\lim _{s \to 0} \frac{G(s)A-A}{s}\;\mbox{ in } \;H_Q\}
\]
and $L(A) = \lim _{s \to 0} \frac{G(s)A-A}{s}$ for any $A \in \dom(L)$.
Notice that $C^1 _c (\R^m) \subset \dom(L)$ and $L(A) = (b \cdny ) A - \dyb A - A \;{^t \dyb}$, $A \in C^1 _c (\R^m)$ (use the hypothesis $Q \in \ltloc{}$ and the dominated convergence theorem). The main properties of the operator $L$ are summarized below (see \cite{BosAnisoDiff} Proposition 3.13 for details)
\begin{pro}
\label{PropOpeL} 
Assume that the hypotheses \eqref{Equ15}, \eqref{Equ17}, \eqref{Equ32Bis}, \eqref{Equ33} hold true.
\begin{enumerate}
\item
The domain of $L$ is dense in $H_Q$ and $L$ is closed.
\item 
The matrix field $A \in H_Q$ belongs to $\dom (L)$ iff there is a constant $C >0$ such that 
\begin{equation*}
\label{Equ61} |G(s)A - A |_Q \leq C |s|,\;\;s \in \R.
\end{equation*}
\item
The operator $L$ is skew-adjoint and we have the orthogonal decomposition $H_Q = \ker L \overset{\perp}{\oplus} \overline{\ran L}$. 
\end{enumerate}
\end{pro}
The transformations $(G(s))_{s\in \R}$ behave nicely also when applied on weighted $L^\infty$ spaces. We introduce the set
\begin{equation}
\label{HQinf}
H_Q ^\infty = \{ A(y) \;\mbox{measurable}\;:\; Q^{1/2}AQ^{1/2} : Q^{1/2}AQ^{1/2} = QA:AQ \in \liy\}.
\end{equation}
It is a Banach space with respect to the norm
\[
|A|_{H_Q^\infty} = \mbox{ess sup}_{y \in \R^m} (Q(y)A(y):A(y)Q(y))^{1/2}.
\]
This space is left invariant by $(G(s))_{s \in \R}$. Indeed, let us consider $A\in H_{Q}^{\infty}$ and, thanks to \eqref{EquQGs} and to the orthogonality of $\calOs$, observe that
\begin{align*}
Q^{1/2} G(s)A Q^{1/2}:Q^{1/2} G(s)A Q^{1/2}  &= \,^{t}\calOs Q^{1/2}_{s} A_{s} Q^{1/2}_{s} \calOs :{}^{t}\calOs Q^{1/2}_{s} A_{s} Q^{1/2}_{s} \calOs \\
&= ( Q^{1/2}\,A\,Q^{1/2} :Q^{1/2}\,A\,Q^{1/2})_s,\quad \;s\in \R. 
\end{align*}
We deduce that for any $s \in \R$ we have $G(s)A \in H_Q^\infty$ and $|G(s)A|_{H_{Q}^{\infty}} = |A|_{H_{Q}^{\infty}}$.\\
We are now in position to apply the von Neumann's ergodic mean theorem.
\begin{thm}(von Neumann's ergodic mean theorem)\\
\label{vonNeumann}
Let $(\mathcal{G}(s))_{s\in \R}$ be a $C^{0}$-group of unitary operators on an Hilbert space $(\mathcal{H}, (\cdot, \cdot))$ and $\mathcal{L}$ be its infinitesimal generator. Then for any $x \in \mathcal{H}$, we have the strong convergence in $\mathcal{H}$
\[
\limS \frac{1}{S}\int_{r}^{r+S}\mathcal{G}(s)x \;\md s = \mathrm{Proj}_{\text{Ker}\mathcal{L}} x ,\quad \text{uniformly with respect to}\;r\in \R.
\]
\end{thm}
The proof of Theorem \ref{AveMatField} comes immediately, by applying Theorem \ref{vonNeumann} to the group in Proposition \ref{CZGroup}.
\begin{proof} (of Theorem \ref{AveMatField}) 
The first and second statements are obvious.\\
3. For any $\xi \in \R^m, \psi \in C^0 _c (\R^m), \psi \geq 0$ we have $\psi (\cdot) P^{1/2} \xi \otimes P^{1/2} \xi \in H_Q$ and we can write, thanks to \eqref{EquQGs}
\begin{align*}
(G(s)A, \psi (\cdot) P^{1/2} \xi \otimes P^{1/2} \xi)_Q & = \inty{Q^{1/2} G(s)AQ^{1/2} : \psi (y) \xi \otimes \xi }\\
& = \inty{\psi(y)\; {^t {\mathcal O}}(s; y) Q_s ^{1/2} A_s Q_s ^{1/2} {\mathcal O}(s; y)\xi \cdot \xi }\\
& = \inty{\psi (y) Q_s ^{1/2} A_s Q_s ^{1/2}: {\mathcal O}(s; y)\xi \otimes {\mathcal O}(s; y)\xi}\\
& \geq \alpha \inty{|{\mathcal O}(s; y)\xi |^2 \psi (y)} \\
& = \alpha |\xi |^2 \inty{\psi (y)}.
\end{align*}
Taking the average over $[0,S]$ and letting $S \to +\infty$ yield 
\begin{align*}
\inty{Q^{1/2} \ave{A} Q^{1/2} : \xi \otimes \xi \psi (y)} & = (\ave{A}, \psi P^{1/2}\xi \otimes P^{1/2}\xi )_Q \nonumber \\
& = \limS \frac{1}{S} \intS{(G(s)A, \psi P^{1/2}\xi \otimes P^{1/2}\xi )_Q} \\
& \geq \inty{\alpha |\xi |^2 \psi (y)}
\end{align*}
implying that 
\[
Q^{1/2}(y) \ave{A} (y)Q^{1/2}(y) \geq \alpha I_m,\;\;y \in \R^m.
\]
4. Obviously, for any $A \in H_Q$, we have by the properties of the orthogonal projection on $\ker L$ that $|\ave{A}|_Q = |\mathrm{Proj} _{\ker L} A|_Q \leq |A|_Q$. For the last inequality, consider $M \in \mathcal{M}_m (\R)$ a fixed matrix, $\psi \in C^0_c (\R^m), \psi \geq 0$ and, as before, observe that $\psi P^{1/2} M P^{1/2} \in H_Q$, which allows us to write
\begin{align*}
(G(s)A, & \psi P^{1/2} M P^{1/2} )_Q  = \inty{Q^{1/2} G(s)A Q^{1/2}: \psi M} \\
& = \inty{\;^t \calO (s;y) Q_s ^{1/2} A_s Q_s ^{1/2} \calO (s;y) : \psi M}\\
& = \inty{Q_s ^{1/2} A_s Q_s ^{1/2} : \calO (s;y) M \;^t \calO (s;y)} \\
& \leq \inty{\sqrt{Q_s ^{1/2} A_s Q_s ^{1/2} : Q_s ^{1/2} A_s Q_s ^{1/2}} \sqrt{\calO (s;y) M \;^t \calO (s;y) : \calO (s;y) M \;^t \calO (s;y)}\psi }\\
& \leq |A|_{H_Q ^\infty} (M:M)^{1/2} \inty{\psi (y)}.
\end{align*}
Taking the average over $[0,S]$ and letting $S \to +\infty$, lead to
\[
\inty{Q^{1/2}\ave{A}Q^{1/2}:M\psi (y) } = (\ave{A}, \psi P^{1/2}MP^{1/2})_Q \leq |A|_{H_Q^\infty} (M:M)^{1/2}\inty{\psi (y)}.
\]
We deduce that 
\[
Q^{1/2}(y) \ave{A}(y) Q^{1/2}(y) : M \leq |A|_{H_Q^\infty} (M:M)^{1/2},\;\;y \in \R^m,\;\;M \in \mathcal{M}_m(\R)
\]
saying that 
\[
|\ave{A}|_{H_Q^\infty} = \mbox{ess sup}_{y \in \R^m} \sqrt{Q^{1/2}(y) \ave{A}(y) Q^{1/2}(y) : Q^{1/2}(y) \ave{A}(y) Q^{1/2}(y)} \leq |A|_{H_Q ^\infty}.
\]
\end{proof}
\noindent
We introduce also the sets of vector fields
\[
X_Q = \{ c:\R^m \to \R^m  \mbox{ measurable}\;:\; \inty{Q(y) : c(y)\otimes c(y) } < +\infty\}
\]
\[
X_Q^\infty = \{ c:\R^m \to \R^m  \mbox{ measurable}\;:\; |Q^{1/2}c|\in \liy \}.
\]
The vector space $X_Q$, endowed with the scalar product 
\[
(\cdot, \cdot)_Q : X_Q \times X_Q \to \R,\;\;(c,d)_Q = \inty{Q(y):c(y) \otimes d(y)},\;\;c, d \in X_Q
\]
becomes a Hilbert space, whose norm is denoted by $|c|_Q = (c, c)_Q^{1/2}, c \in X_Q$. We use the same notation for the scalar product and norm of $H_Q$, resp. $X_Q$. Obviously, it should be understood in the right framework, depending on the arguments being matrix fields, resp. vector fields.\\
The vector space $X_Q ^\infty$ is a Banach space with respect to the norm 
\[
|c|_{X_Q^\infty} = \mbox{ess sup}_{y \in \R^m} |Q^{1/2}(y)c(y)|.
\] 
We end this section by indicating a sufficient condition for \eqref{Equ32Bis}, \eqref{Equ33}. Assume that there is a matrix field $R(y)$ such that 
\begin{equation}
\label{E41} \det R(y) \neq 0,\;y \in \R^m,\;\;R \in \loloc{}
\end{equation}
\begin{equation}
\label{E42}
(b \cdny )R + R \dyb = 0 \;\mbox{in } \dpri.
\end{equation}
The hypothesis \eqref{E42} is equivalent to $R(\ysy) \partial  \ysy  =R(y),\;\;(s, y) \in \R \times \R^m$, which also writes
\begin{equation}
\label{E43}
\partial  \ysy R^{-1} (y) = R^{-1} (\ysy),\;\;(s,y)\in \R \times \R^m.
\end{equation}
We deduce that \eqref{E42} is equivalent to $(b \cdny ) R^{-1} = \dyb R^{-1}$, saying that the columns of $R ^{-1}$ are vector fields in involution with $b$. The vector fields in the columns of $R^{-1}$ are denoted $b_i, 1\leq i\leq m$. At any point $y \in \R^m$ they form a basis for $\R^m$ cf. \eqref{E41} and are supposed smooth
\begin{equation}
\label{Equ58}
b_i \in W^{1, \infty}_{\mathrm{loc}}(\R^m),\;\;\divy b_i \in \liy{},\;\;1\leq i\leq m.
\end{equation}
We assume that any field $b_i$ satisfies the growth condition 
\begin{equation}
\label{Equ58Bis}
\forall i \in \{1, ..., m\}, \;\exists C_i >0 \mbox{ such that } |b_i(y)|\leq C_i ( 1 + |y|),\;\;y \in \R^m
\end{equation}
which guarantees the existence of the global flows $Y_i(s;y) \in W^{1, \infty}_{\mathrm{loc}}(\R \times \R^m),  i \in \{1, ..., m\}$. Clearly $R^{-1} \in L^\infty _{\mathrm{loc}}(\R^m)$, since $b_i$, which are the columns of $R^{-1}$, are supposed  locally bounded on $\R^m$. Since $y \to R^{-1}(y)$ is continuous, the function $y \to \det R^{-1} (y)$ remains away from $0$ on any compact set of $\R^m$, implying that $R = (R^{-1})^{-1} \in \lilocy{}$. In particular $^t R R, (^t R R)^{-1}$ are locally bounded, and therefore locally square integrable on $\R^m$. We define $Q = \;^t R R, P = Q^{-1} = R^{-1} \;^t R^{-1}$ and observe that \eqref{Equ32Bis}, \eqref{Equ33} are satisfied. Indeed, $P(y)$ is symmetric, definite positive, locally square integrable, together with its inverse $Q = P^{-1}$ and, thanks to \eqref{E43}, we have
\begin{align*}
P(\ysy) & = R^{-1} (\ysy) \;{^t R}^{-1} (\ysy) \\
& = \partial  \ysy R^{-1} (y) \;{^t R} ^{-1}(y) \;{^t \partial  \ysy} \\
& = \partial  \ysy P(y) \;{^t \partial  \ysy}
\end{align*}
saying that $[b,P] = 0$ in $\dpri$ cf. Proposition \ref{CaracBraMat}. Under the hypotheses \eqref{Equ58}, \eqref{Equ58Bis}, the space $H_Q, H_Q ^\infty$ also write
\[
H_Q = \left \{A:\R^m \to \mathcal{M}_m(\R) \mbox{ measurable}\;:\; \inty{R(y)A(y)\;^tR(y):R(y)A(y)\;^tR(y)}<+\infty \right \}
\]
and
\[
H_Q ^\infty = \{A:\R^m \to \mathcal{M}_m(\R)\mbox{ measurable}\;:\; R(y)A(y)\;^tR(y):R(y)A(y)\;^tR(y)\in \liy{}\}.
\]
Given the family $(b_i)_{1\leq i \leq m}$ of vector fields in involution with respect to $b$, we construct the following $H^1$ type space on $\R^m$
\begin{equation}
\label{HoneR}
H^1 _R = \bigcap _{i = 1} ^m \dom (b_i \cdot \nabla _y) = \{u \in \lty{}\;:\; {^t }R^{-1} \nabla _y u := \;^t (b_1\cdot \nabla _y u, ..., b_m \cdot \nabla _y u ) \in \lty{} ^m \}
\end{equation}
endowed with the scalar product
\[
(u, v)_R = \inty{u(y) v(y)} + \sum _{i = 1} ^m \inty{(b_i \cdot \nabla _y u) (b_i \cdot \nabla _y v)},\;u, v \in H^1_R.
\]
It is a Hilbert space, whose norm is denoted by $|\cdot |_R$. The operators $b_i \cdny $ are the infinitesimal generators of the $C^0$-groups of linear transformations on $\lty{}$ given by
\[
\tau _i (s) u = u \circ Y_i (s;\cdot),\;\;u \in \lty,\;\;s\in \R,\;\;i \in \{1,...,m\}.
\]
The hypothesis $\divy b_i \in \liy{}$ plays a crucial role when looking for a bound for the Jacobian determinant of $\partial Y_i$. 
\begin{remark}
\label{WeakGradient}
Notice that every element of $H^{1}_{R}$ has a weak gradient in $L^{2}_{loc}(\R^{m})$. Indeed, if $u\in H^{1}_{R}$, we define $v_{i} = b_{i}\cdot \nabla_{y}u$, $i \in \{1, ...,m\}$ and consider $V(y) = {}^{t}R(y)\,{}^{t}(v_{1}(y),...,v_{m}(y))$, $y\in \R^{m}$. The field $V$ is locally square integrable on $\R^m$ since $R$ is locally bounded on $\R^m$ and $(v_i)_{1\leq i \leq m}$ are square integrable on $\R^m$. Using the dual basis $\{c_1, ..., c_m\}$ of $\{b_1, ...,b_m\}$ we write for any $\xi \in (C^1 _c (\R^m))^m$
\begin{align*}
\inty{V(y) \cdot \xi (y)}  & = \inty{\sum _{i = 1} ^m (V(y) \cdot b_i (y)) \;(c_i (y) \cdot \xi (y))}\\
& = - \inty{u(y) \divy \left ( \sum _{i = 1} ^m (c_i(y) \cdot \xi (y))b_i (y) \right )} = - \inty{u(y) \divy \xi }
\end{align*}
which shows that $V$ is the weak gradient of $u$. The $H^1_R$ norm of $u$ can be written using the $X_P$ norm of the gradient
\begin{align}
\label{E47}
|u|^2 _R & = \|u \|^2 _{\lty{}} + \|{^t R } ^{-1} \nabla _y u \|^2 _{\lty{}} \\
& = \inty{\{(u(y))^2 + ( R^{-1} \;{^t R }^{-1} V \cdot V)\}} \nonumber \\
& = \inty{\{(u(y))^2 + (P(y)V(y)\cdot V(y))\}} \nonumber \\
& = \|u \|^2 _{\lty{}} + |V|^2 _P.\nonumber 
\end{align} 
\end{remark} 
In the sequel, for any $u \in H^1_R$, the notation $\nabla _y u$ stands for the weak gradient of $u$, and we have $\nabla _y u = \;^t R \;^t (b_1 \cdny u,...,b_m \cdny u)$. We introduce the differential operator 
\[
\nry := \;^t R^{-1} \nabla _y = \;^t (b_1 \cdny, ..., b_m \cdny).
\]
We will also use the space
\begin{align}
\label{HtwoR}
H^2 _R  = \{u \in H^1 _R\;:\; \nry u \in (H^1_R)^m\} 
 = \{ u \in H^1_R\;:\; b_j\cdny (b_i \cdny u) \in \lty{}, 1\leq i, j\leq m\} 
\end{align}
and the differential operator $(\nry )^2 := \nry \otimes \nry $ given by
\[
(\nry \otimes \nry)_{ij} = b_j \cdny (b_i \cdny ),\;\;i, j \in \{1, ..., m\}.
\]

\subsection{Examples}

In this paragraph we compute explicitly the average matrix field in two cases. Both of them deal with periodic flows. Consider the vector field $b(y)=(\gamma y_2, - \beta y_1)$, for any $y = (y_1, y_2) \in \R^2$, with $\beta, \gamma \in \R_+ ^\star$. We denote by $Y(s;y)$ the flow of the vector field $b$. We intend to determine the average along the flow $Y$ of the matrix field
\[
D(y) = \left ( 
\begin{array}{ccc}
\lambda _1 (y) & 0\\
0 & \lambda _2 (y)
\end{array}
\right),\;\;y \in \R^2
\]
where $\lambda _1, \lambda _2$ are two given functions. It is easily seen that the flow is $2\pi /\sqrt{\beta \gamma}$-periodic and writes $Y(s;y) = \mathcal{R}(-s;\beta, \gamma)y,\;(s,y) \in \R \times \R^2$, with
\[
\mathcal{R}(s;\beta, \gamma) = \left ( 
\begin{array}{ccc}
\cos ( \sqrt{\beta \gamma} \,s)  & - \sqrt{\frac{\gamma}{\beta}} \sin ( \sqrt{\beta \gamma} \,s)\\
\sqrt{\frac{\beta}{\gamma}} \sin ( \sqrt{\beta \gamma}\, s) & \cos ( \sqrt{\beta \gamma}\, s)
\end{array}
\right).
\]
By Theorem \ref{AveMatField} we deduce that
\begin{align*}
\ave{D} & = \frac{\sqrt{\beta \gamma}}{2\pi} \int _0 ^{2\pi/\sqrt{\beta \gamma}} \partial Y (-s; \ys)D (\ys) \,^t \partial Y (-s; \ys)\;\mathrm{d}s\\
& = \frac{\sqrt{\beta \gamma}}{2\pi} \int _0 ^{2\pi/\sqrt{\beta \gamma}} \mathcal{R}(s;\beta, \gamma) D(\ys) \mathcal{R}(-s;\gamma, \beta)\;\mathrm{d}s \\
& = \left ( 
\begin{array}{ccc}
\ave{D}_{11}  &  \ave{D}_{12}\\
\ave{D}_{21}  &  \ave{D}_{22}
\end{array}
\right )
\end{align*}
where 
\[
\ave{D}_{11} = \frac{1}{2}\ave{\lambda_1 [ 1 + \cos (2 \sqrt{\beta \gamma }\,\cdot)]} + \frac{\gamma}{2\beta} \ave{\lambda _2 [1 - \cos (2 \sqrt{\beta \gamma }\,\cdot)]}
\]
\[
\ave{D}_{12} = \ave{D}_{21} = \frac{\sqrt{\beta}}{2 \sqrt{\gamma}}\ave{\lambda _1 \sin (2 \sqrt{\beta \gamma }\,\cdot)} - \frac{\sqrt{\gamma}}{2 \sqrt{\beta}}\ave{\lambda _2 \sin (2 \sqrt{\beta \gamma }\,\cdot)}
\]
\[
\ave{D}_{22} = \frac{\beta}{2\gamma}\ave{\lambda_1 [ 1 - \cos (2 \sqrt{\beta \gamma }\,\cdot)]} + \frac{1}{2} \ave{\lambda _2 [1 - \cos (2 \sqrt{\beta \gamma }\,\cdot)]}
\]
and for any function $h$ the notation $\ave{\lambda _i h(\cdot)}$ stands for $\frac{\sqrt{\beta \gamma}}{2\pi} \int _0 ^{2\pi/\sqrt{\beta \gamma}} \lambda _i (\ysy) h(s)\;\mathrm{d}s,\; i \in \{1, 2\}$. Notice that when $\lambda _1, \lambda_2$ are constant functions along the flow $\ys$ (that is, when $\lambda _1, \lambda _2$ depend only on $\beta (y_1)^2 + \gamma (y_2)^2$), the expression for $\ave{D}$ reduces to
\[
\ave{D} = \left ( 
\begin{array}{ccc}
\frac{1}{2}\lambda _1 + \frac{\gamma}{2\beta} \lambda _2 & 0\\
0 & \frac{\beta}{2\gamma}\lambda _1 + \frac{1}{2} \lambda _2
\end{array}
\right).
\]
We inquire now about the Fokker-Planck equation. The external electro-magnetic field is given by
\[
E = - \nabla _x \Phi,\;\;B^\eps = (0, 0, B/\eps),\;\;x\in \R^3
\]
where, for simplicity, we assume that the magnetic field is uniform. In the finite Larmor radius regime ({\it i.e.,} the typical length in the orthogonal directions is much smaller then the typical length in the parallel direction), the presence density $\fe$ satisfies
\[
\partial _t \fe + \frac{1}{\eps}( v_1 \partial _{x_1}  + v_2 \partial _{x_2}) \fe + v_3 \partial _{x_3} \fe + \frac{q}{m} E \cdot \nabla _v \fe + \frac{qB}{m\eps} (v_2 \partial _{v_1} - v_1 \partial _{v_2}) \fe = \nu \mathrm{div}_v \{\Theta \nabla _v \fe + v\fe\}.
\]
Here $m$ is the particle mass, $q$ is the particle charge, $\nu$ is the collision frequency and $\Theta$ is the temperature. In this case, the flow $Y(s;x,v)$ to be considered corresponds to the vector field
\[
b(x,v) \cdot \nabla _{x,v} =  v_1 \partial _{x_1}  + v_2 \partial _{x_2} + \omega _c (v_2 \partial _{v_1} - v_1 \partial _{v_2}),\;\;\omega_c = \frac{qB}{m},\;\;(x,v) \in \R^6.
\]
It is easily seen that 
\[
\olX(s;\olx, \olv) = \olx + \frac{\polv}{\oc}- \frac{\calr(-\oc s)}{\oc} \;\polv,\;\;X_3(s;x_3) = x_3,\;\;\olV (s; \olv) = \calr (-\oc s) \olv,\;\;V_3(s;v_3) = v_3
\]
where we have used the notations $\olx = (x_1, x_2), \olv = (v_1, v_2), \polv = (v_2, - v_1)$ and $\calr (\theta)$ stands for the rotation of angle $\theta \in \R$. The Jacobian matrix writes
\[
\partial _{x,v} Y (s; x, v) = \left (
\begin{array}{ccccc}
I_2 \quad & O_{2\times 1} \quad & \frac{I_2 - \calr (-\oc s)}{\oc} \cale \quad &  O _{2 \times 1} \\
O_{1\times2} \quad & 1 \quad & O_{1\times 2} \quad & 0 \\
O_{2\times2} \quad & O_{2\times1} \quad & \calr (-\oc s) \quad & O_{2\times 1} \\
O_{1\times2} \quad & 0 \quad & O_{1\times 2} \quad & 1 \\
\end{array}
\right)
\]
where $O_{m\times n}$ stands for the null matrix with $m$ lines and $n$ columns, and $\cale = \calr (-\pi/2)$. The diffusion matrix field to be averaged is 
\[
D = \sum _{i = 1}^3 e_{v_i} \otimes e_{v_i} = 
\left ( 
\begin{array}{ll} 
{O_3} \quad & {O_3} \\
{O_3} \quad & {I_3} 
\end{array}
\right)
\]
and by Theorem \ref{AveMatField} we obtain after direct computations
\begin{align}
\label{EquFPDiffusion}
\ave{D}  = \ave{
\!\!\left (
\begin{array}{ll} 
{O_3} \quad & {O_3} \\
{O_3} \quad & {I_3} 
\end{array}
\right )\!\!} & = \!\! \frac{\oc}{2\pi}\int _0 ^{2\pi/\oc} \partial Y(-s;\ys) \!\!\left (
\begin{array}{ll} 
{O_3} \quad & {O_3} \\
{O_3} \quad & {I_3} 
\end{array}
\right )\;{^t \partial} Y(-s;\ys)\;\mathrm{d}s \nonumber \\
& = \left (
\begin{array}{cccc}
\frac{2I_2}{\oc ^2}  \quad & O_{2\times 1} \quad & - \frac{\cale}{\oc} \quad & O_{2\times 1}\\
O_{1\times 2} \quad & 0 \quad & O_{1\times 2} \quad & 0\\
\frac{\cale}{\oc} \quad & O_{2\times 1} \quad & I_2 \quad & O_{2\times 1}\\
O_{1\times 2} \quad & 0 \quad & O_{1\times 2} \quad & 1
\end{array}
\right).
\end{align}
Notice that the average Fokker-Planck kernel contains diffusion terms not only in velocity variables (as in the Fokker-Planck kernel) but also in space variables (orthogonal to the magnetic lines), as observed in gyrokinetic experiments and numerical simulations.

\section{Well posedness for the perturbed problem and uniform estimates}
\label{UnifEst}

For solving \eqref{Equ1}, we appeal to variational methods. We use the continuous embedding $H^1_R \hookrightarrow \lty$, with dense image (since $C^1 _c (\R^m) \subset H^1 _R$). We work under the hypotheses \eqref{Equ15}, \eqref{Equ17}, \eqref{Equ58}, \eqref{Equ58Bis}. Moreover we assume that 
\begin{equation}
\label{E45}
^t D = D,\;\;D \in H_Q \cap H_Q ^\infty,\;\;b \in X_Q ^\infty
\end{equation}
and
\begin{equation}
\label{E46}
\exists \;\alpha >0 \;\mbox{ such that } Q^{1/2} (y) D(y) Q^{1/2}(y) \geq \alpha I_m,\;\;y \in \R^m
\end{equation}
where $Q = \;^t R R$ and the columns of $R^{-1}$ are given by the vector fields $b_1, ..., b_m$. 
\begin{pro}
\label{Coercivity}
Assume that the hypotheses \eqref{Equ15}, \eqref{Equ17}, \eqref{Equ58}, \eqref{Equ58Bis}, \eqref{E45}, \eqref{E46} hold true. 
\begin{enumerate}
\item
Let us consider the application $a^\eps : H^1 _R \times H^1 _R \to \R$
\[
a^\eps (u,v) = \inty{D(y)\nabla _y u \cdot \nabla _y v} + \frac{1}{\eps} \inty{(b \cdny u ) v(y)},\;\;u, v \in H^1 _R.
\]
The bilinear form $a^\eps$ is well defined, continuous and coercive on $H^1_R$ with respect to $\lty$.
\item
Let us consider the application $\ave{a} : H^1 _R \times H^1 _R \to \R$
\[
\ave{a} (u,v) = \inty{\ave{D}(y) \nabla _y u \cdot \nabla _y v},\;\;u, v \in H^1 _R.
\]
The bilinear form $\ave{a}$ is well defined, continuous and coercive on $H^1_R$ with respect to $\lty$.
\end{enumerate}
\end{pro}
\begin{proof}$\;$\\
1. For any $u, v \in H^1_R$ we have
\begin{align*}
|D(y) \nabla _y u \cdot \nabla _y v | & = | Q^{1/2}(y)D(y) Q^{1/2}(y) : (P^{1/2} \nabla _y v ) \otimes (P^{1/2} \nabla _y u) | \\
& \leq |D|_{H_Q ^\infty} |P^{1/2} (y) \nabla _y v| \;|P^{1/2}(y) \nabla _y u|,\;\;y \in \R^m
\end{align*}
and
\[
|b(y) \cdny u \;v(y) | = |Q^{1/2} (y) b(y) \cdot P^{1/2}(y) \nabla _y u\;v(y) | \leq |b|_{X_Q ^\infty} |P^{1/2} (y) \nabla _y u |\;|v(y) |,\;\;y \in \R^m.
\]
Therefore, it is easily seen, thanks to \eqref{E47}, that 
\begin{align*}
|a^\eps (u,v) | & \leq |D|_{H_Q ^\infty} |\nabla _y u |_P |\nabla _y v |_P + \frac{1}{\eps} |b|_{X_Q ^\infty} |\nabla _y u |_P \|v \|_{\lty{}} \\
& \leq \left ( |D|_{H_Q ^\infty} + \frac{1}{\eps} |b|_{X_Q ^\infty}\right ) |u|_R |v|_R
\end{align*}
saying that the bilinear application $a^\eps(\cdot, \cdot)$ is well defined and continuous. We inquire now about the coercivity of $a^\eps$ on $H^1 _R$, with respect to $\lty$. For any $u \in H^1_R$ we have, thanks to the anti-symmetry of $b\cdny $
\begin{align*}
a^\eps (u,u) & + \alpha \|u \|^2 _{\lty{}}  = \inty{D(y) \nabla _y u \cdot \nabla _y u } + \alpha \|u \|^2 _{\lty{}} \\
& = \inty{Q^{1/2}(y) D(y) Q^{1/2}(y) : (P^{1/2}(y) \nabla _y u ) \otimes (P^{1/2}(y) \nabla _y u )} + \alpha \|u \|^2 _{\lty{}}\\
& \geq \alpha \inty{|P^{1/2}(y) \nabla _y u |^2 } + \alpha \|u \|^2 _{\lty{}} \\
& = \alpha \left ( |\nabla _y u |_P ^2 + \|u \|^2 _{\lty{}} \right )= \alpha |u|^2 _R.
\end{align*}
We emphasize the following inequality, which will be used several times in the sequel
\begin{equation}
\label{EquCoerH1R}
D(y)\nabla _y u \cdot \nabla _y u \geq \alpha |\nry u |^2,\;\;u \in H^1_R.
\end{equation}
2. Follow the same lines as before by using the third and fourth assertions of Theorem \ref{AveMatField}, that is
\[
Q^{1/2}(y) \ave{D}(y) Q^{1/2}(y) \geq \alpha I_m,\;\;y \in \R^m
\]
and $|\ave{D}|_{H_Q^\infty} \leq |D|_{H_Q^\infty}$. 
\end{proof}
\begin{pro}
\label{GenEst}
Assume that the hypotheses \eqref{Equ15}, \eqref{Equ17}, \eqref{Equ58}, \eqref{Equ58Bis}, \eqref{E45}, \eqref{E46} hold true. There exists $\ue$ (resp. $v$) a unique variational solution of \eqref{Equ1} (resp. \eqref{Equ14}). Moreover,  we have
\[
\|\ue\|_{\litlty} \leq \|\uin \|_{\lty},\;\;\|\nabla _y \ue \|_{\lttxp} \leq \frac{\|\uin \|_{\lty}}{\sqrt{2\alpha}},\;\;\eps >0
\] 
and
\[
\|v\|_{\litlty} \leq \|\uin \|_{\lty},\;\;\|\nabla _z v \|_{\lttxp} \leq \frac{\|\uin \|_{\lty}}{\sqrt{2\alpha}}.
\] 
\end{pro}
\begin{proof}
By Theorems 1, 2 of \cite{DauLions00} p. 513, see also \cite{LioMag68}, we deduce that for any $\uin \in \lty$, there is a unique variational solution $\ue$ for the problem \eqref{Equ1}, that is $\ue \in C_b(\R_+;\lty) \cap L^2(\R_+;H^1_R)$, $\partial _t \ue \in L^2(\R_+; (H^1_R)^\prime)$ and
\[
\ue (0) = \uin,\;\;\frac{\md }{\md t} \inty{\ue (t,y) \varphi (y)} + a^\eps (\ue(t), \varphi) = 0,\;\;\mbox{in } \mathcal{D}^\prime (\R_+),\mbox{ for any } \varphi \in H^1_R.
\]
Similarly, there is a unique variational solution $v$ for the limit model \eqref{Equ14}, that is $v \in C_b(\R_+;\lty) \cap L^2(\R_+;H^1_R)$, $\partial _t v \in L^2(\R_+; (H^1_R)^\prime)$ and
\[
v (0) = \uin,\;\;\frac{\md }{\md t} \inty{v (t,z) \psi (z)} + \ave{a} (v(t), \psi) = 0,\;\;\mbox{in } \mathcal{D}^\prime (\R_+),\mbox{ for any } \psi \in H^1_R.
\]
The above estimates come immediately by the energy balance
\[
\frac{1}{2} \frac{\md }{\md t} \|\ue (t) \|^2 _{\lty{}} + a^\eps (\ue (t), \ue (t)) = 0,\;\;\mbox{ in } \mathcal{D}^\prime (\R_+)
\]
which implies 
\[
\frac{1}{2}\|\ue (t) \|^2 _{\lty{}} + \int _0 ^t a^\eps (\ue (\tau), \ue (\tau))\;\md \tau = \frac{1}{2} \|\uin \|^2 _{\lty{}},\;\;t \in \R_+.
\]
In particular we deduce $\|\ue (t) \|_{\lty} \leq \|\uin \|_{\lty}$, for any $t \in \R_+, \eps >0$, and 
\[
2\alpha \int _0 ^t |\nabla _y \ue (\tau) |_P ^2 \;\md \tau \leq 2 \int _0 ^t a^\eps (\ue (\tau), \ue (\tau))\;\md \tau \leq \|\uin \|^2 _{\lty{}},\;\;t \in \R_+,\;\;\eps >0
\]
saying that $\|\nabla _y \ue \|_{\lttxp} \leq \frac{\|\uin \|_{\lty}}{\sqrt{2\alpha}}$, for any $\eps >0$. The estimates for $v$ follow similarly, using the energy balance
\[
\frac{1}{2} \frac{\md }{\md t} \|v(t) \|^2 _{\lty} + \ave{a} (v(t), v(t)) = 0,\;\;\mbox{in } \mathcal{D}^\prime (\R_+)
\]
and the inequality $\ave{a} (v(t), v(t)) \geq \alpha |\nabla _z v(t) |^2 _P$. 
\end{proof}
\begin{remark}
\label{Estveps}
The family $(\ve(t,\cdot))_{\eps >0} = (\ue (t, Y(t/\eps;\cdot)))_{\eps>0}$ satisfies the same estimates as the family $(\ue)_{\eps>0}$. Indeed, performing the change of variable $y = \ytz$, which is measure preserving, one gets
\[
\|\ve (t)\|^2 _{\lty} = \intz{(\ue (t, \ytz))^2} = \inty{(\ue (t,y))^2} \leq \|\uin \|^2 _{\lty},\;\;t \in \R_+,\;\;\eps >0.
\]
Notice that we have, thanks to \eqref{E43}
\begin{align*}
\nrz \ve (t) & = \;^t R^{-1} (z) \nabla _z \ve (t) = \;^t R^{-1}(z) \;^t \partial \ytz \nabla _y \ue (t, \ytz) \\
& = \;^t ( \partial \ytz R^{-1} (z)) \nabla _y \ue (t, \ytz) \\
& = \;^t R^{-1} (\ytz) \nabla _y \ue (t, \ytz) = \left ( \nry \ue (t)\right) (\ytz)
\end{align*}
and therefore 
\begin{align*}
\|\nabla _z \ve \|^2 _{\lttxp} & = \int _0 ^{+\infty} |\nabla _z \ve (\tau)|^2_P \;\md \tau = \int _0 ^{+\infty} \|\nrz \ve (\tau)\|^2 _{\lty} \;\md \tau \\
& = \int _0 ^{+\infty} \|\nry \ue (\tau) \|^2 _{\lty} \;\md \tau = \int _0 ^{+\infty}|\nabla _y \ue (\tau )|^2_P \;\md \tau = \|\nabla _y \ue \|^2 _{\lttxp}\\
& \leq \frac{\|\uin \|^2 _{\lty}}{2\alpha},\;\;\eps >0.
\end{align*}
Using twice the formula
\[
b_i \cdot \nabla _z \ve (t) = b_i \cdot \nabla _z ( \ue (t) \circ \yt) =  ( b_i \cdny \ue (t)) \circ \yt
\]
we deduce that 
\[
b_j \cdot \nabla _z ( b_i \cdot \nabla _z \ve (t))  = b_j \cdot \nabla _z \left [ ( b_i \cdny \ue (t)) \circ \yt\right ] = \left [ b_j \cdny (b_i \cdny  \ue (t))\right] \circ \yt.
\]
Therefore $\|b_j \cdot \nabla _z (b_i \cdot \nabla _z \ve (t))\|_{\lty} = \|b_j \cdny ( b_i \cdny \ue (t))\|_{\lty}$, $i, j \in \{1,...,m\}$ anytime that $\ue (t) \in H^2_R$. 
\end{remark}
Up to now, we have considered solutions with initial condition $\uin \in \lty$. In order to study the stability of the family $(\ve)_{\eps>0}$ when $\eps$ goes to $0$, we need more regularity. This will be the object of the next propositions, in which we analyze how the regularity of the initial condition propagates in time. The idea is to take the directional derivative $b_i \cdny $ of \eqref{Equ1},  leading to
\begin{equation}
\label{E48}
\partial _t (b_i \cdny \ue) - \divy (D(y) \nabla _y (b_i \cdny \ue)) + \frac{1}{\eps} b \cdny (b_i \cdny \ue) = [b_i \cdny, \divy (D \nabla _y )]\ue.
\end{equation}
Notice that the key point was to take advantage of the involution between $b_i$ and $b$, for any $i \in \{1, ...,m\}$, which guarantees that there is no commutator between the first order operators $b_i \cdny$ and $b \cdny$. More generally, if we apply the directional derivative $c \cdny$ in \eqref{Equ1}, the right hand side of the corresponding equation in \eqref{E48} will contain the extra term $\frac{1}{\eps} [b\cdny, c \cdny ]\ue$, which is clearly unstable, when $\eps$ goes to $0$, if $b$ and $c$ are not in involution. The estimate for $b_i \cdny \ue$ follows by using the energy balance of \eqref{E48}, observing that, thanks to the anti-symmetry of $b \cdny$, we get rid of the term of order $1/\eps$. We assume that for any $i, j \in \{1, ..., m\}$, the coordinates of the Poisson bracket  $[b_i, b_j]$ in the basis $(b_k)_{1\leq k \leq m}$ are bounded
\begin{equation}
\label{E49}
[b_i, b_j] = \sum _{k = 1}^m \alpha _{ij} ^k b_k,\;\;\alpha _{ij} ^k \in \liy{},\;i,j,k \in \{1, ...,m\}.
\end{equation} 
\begin{pro}
\label{MoreEstimates}
Assume that the hypotheses \eqref{Equ15}, \eqref{Equ17}, \eqref{Equ58}, \eqref{Equ58Bis}, \eqref{E45}, \eqref{E46}, \eqref{E49} hold true. Moreover we assume that for any $i,j \in \{1,...,m\}$
\[
b_i \cdny \divy b_j \in \liy,\;\;\divy (RD) \in \liy{}
\]
\[
R[b_i,D]\;^t R \in \liy{},\;\;\sum _{i = 1} ^m b_i \cdny (R [b_i,D]\;^t R) \in \liy.
\]
If the initial condition belongs to $H^1_R$, then we have for any $T \in \R_+$
\[
\supe \|\nry \ue \|_{L^\infty([0,T];\lty)} = \supe \|\nrz \ve \|_{L^\infty([0,T];\lty)} <+\infty
\]
\[
\supe \|\nry \otimes \nry \ue \|_{L^2([0,T];\lty)} = \supe \|\nrz \otimes \nrz\ve \|_{L^2([0,T];\lty)} <+\infty
\]
\[
\supe \|\partial _t \ve \|_{L^2([0,T];\lty)} <+\infty.
\]
Here the notation $\nabla ^R \otimes \nabla ^R w$ stands for the matrix whose entry $(i,j)$ is $b _j \cdot \nabla (b_i \cdot \nabla w)$, $i,j \in \{1, ...,m\}$. 
\end{pro}
\begin{proof}
We want to estimate the $L^2$ norms of $b_i \cdny \ue$, $i \in \{1, ...,m\}, \eps >0$. This can be done by analyzing the translations along the flows $Y_i$ and estimating the $L^2$ norms of $( \ue (t,Y_i(h;y)) - \ue (t,y))/h$ uniformly with respect to $h \in \R^\star$ and $\eps >0$. For simplicity, we justify the estimates only for smooth solutions and coefficients (and therefore we use clasical derivatives). The general case is left to the reader. We need to compute the commutator between a first order operator $c \cdny$ and the diffusion operator $\divy (D \nabla _y)$. Here $c(y)$ is a vector field, not necessarily in involution with the vector field $b(y)$ (the involution does not play any role when computing $[c\cdny, \divy (D \nabla _y)]$). A straightforward  computation shows that the commutator between $c \cdny$ and $\divy$ is given by
\[
[c\cdny, \divy] \xi = \xi \cdny \divy c - \divy (\partial _y c \;\xi),\;\;\xi \in (C^2 (\R^m))^m.
\]
Using the above formula with $\xi = D(y) \nabla _y \ue$, one gets
\begin{align}
\label{E50}
c\cdny (\divy (D(y) \nabla _y \ue )) - \divy (c\cdny (D(y) \nabla _y \ue)) & = D(y) \nabla _y \ue \cdot \nabla _y \divy c \\
& - \divy (\partial _y c \,D(y) \nabla _y \ue).\nonumber 
\end{align}
Taking into account that 
\begin{align*}
c \cdny (D(y) \nabla _y \ue)  & = (c \cdny D) \nabla _y \ue + D(y)(\partial ^2 \ue) c(y) \\
& = (c \cdny D) \nabla _y \ue + D(y) \nabla _y ( c \cdny \ue) - D(y) \;^t \partial _y c \nabla _y \ue
\end{align*}
we deduce by \eqref{E50}
\begin{align*}
c\cdny (\divy (D(y) \nabla _y \ue)) & - \divy (D(y) \nabla _y (c\cdny \ue)) = D(y) \nabla _y \ue \cdot \nabla _y \divy c\\
& + \divy ( ( c \cdny D - \partial _y c D(y)  - D(y) \;^t \partial _y c ) \nabla _y \ue ).
\end{align*}
Finally the commutator between $c \cdny $ and $\divy (D(y) \nabla _y)$ writes
\begin{equation}
\label{ComFor}
[c \cdny, \divy (D(y) \nabla _y)]\ue = \divy ([c,D]\nabla _y \ue ) + D(y) \nabla _y \ue \cdot \nabla _y \divy c.
\end{equation}
Multiplying \eqref{E48} by $b_i \cdny \ue$, integrating with respect to $y$ over $\R^m$ and observing that the contribution of the singular term $\frac{1}{\eps} b\cdny (b_i \cdny \ue)$ cancels by the anti-symmetry of the operator $b\cdny $, yield
\begin{align}
\label{E51}
\frac{1}{2}\frac{\md }{\md t} \inty{(b_i \cdny \ue  (t))^2 } & + \inty{D(y) \nabla _y (b_i \cdny \ue (t)) \cdny (b_i \cdny \ue (t))} \\
= & - \inty{[b_i,D]\nabla _y \ue (t) \cdot \nabla _y (b_i \cdny \ue (t))} \nonumber \\
& + \inty{D(y) \nabla _y \ue (t) \cdny \divy b_i \;(b_i \cdny \ue (t))}.\nonumber
\end{align}
By hypothesis \eqref{E46} we have
\begin{align}
\label{E52}
D \nabla _y ( b_i \cdny \ue (t)) \cdny ( b_i \cdny \ue (t)) & = Q^{1/2}DQ^{1/2} \\
& : P^{1/2} \nabla _y (b_i \cdny \ue (t)) \otimes P^{1/2} \nabla _y (b_i \cdny \ue (t))\nonumber \\
& \geq \alpha |P^{1/2} \nabla _y (b_i \cdny \ue (t))|^2 = \alpha |\nry (b_i \cdny \ue (t))|^2 \nonumber \\
& = \alpha \sum _{j = 1} ^m (b_j \cdny (b_i \cdny \ue (t)))^2.\nonumber 
\end{align}
Combining \eqref{E51}, \eqref{E52} leads to
\begin{align}
\label{2}
&\frac{1}{2}\frac{\md}{\md t}\left\| \nry \ue(t)\right\|^{2}_{L^{2}(\R^{m})} +\alpha\,\left\| \nry \otimes\nry \ue(t) \right\|_{L^{2}(\R^{m})}^{2} \\
& \leq -\sum_{i=1}^{m}\underbrace{\inty{[b_{i},\,D]\nabla_{y}\ue(t)\cdot \nabla_{y}(b_{i}\cdot\nabla_{y}\ue(t))}}_{:=I^{1}_{i}} \nonumber \\
& + \sum_{i=1}^{m}\underbrace{\inty{D\nabla_{y}\ue(t)\cdot \nabla_{y}(\divy b_{i})\,(b_{i}\cdot\nabla_{y}\ue(t))}}_{:=I^{2}_{i}}\nonumber.
\end{align}
In order to upper bound the term $I^{1}_{i}$ in the right hand side of \eqref{2} we write 
\begin{align}
\label{E53}
I_{i}^{1}=&\int_{\mathbb{R}^{m}}R\,[b_{i},\,D]\,{}^{t}R : {}^{t}R^{-1}\nabla_{y}u^{\eps}(t) \otimes {}^{t}R^{-1}\nabla_{y}(b_{i}\cdot\nabla_{y}u^{\eps}(t))\,\text{d}y\\
=&\int_{\mathbb{R}^{m}}R\,[b_{i},\,D]\,{}^{t}R : \nry u^{\eps}(t) \otimes \nry(b_{i}\cdot\nabla_{y}u^{\eps}(t))\,\text{d}y.\nonumber 
\end{align}
Notice that for any $j \in \{1,...,m\}$ we have
\begin{align*}
b_j \cdny (b_i \cdny \ue (t)) & = b_i \cdny (b_j \cdny \ue (t)) - [b_i,b_j]\cdny \ue (t) \\
& = b_i \cdny (b_j \cdny \ue (t)) - \sum _{k=1} ^m \alpha _{ij}^k b_k \cdny \ue (t)
\end{align*}
and therefore we obtain
\begin{equation}
\label{E54}
\nry(b_{i}\cdot\nabla_{y}\ue(t)) = b_{i}\cdot\nabla_{y}(\nry \ue (t))-\mathcal{A}_i \nry \ue (t)
\end{equation}
where the entry $(j,k)$ of the matrix $\mathcal{A}_i$ is given by $\alpha _{ij} ^k$. Combining \eqref{E53}, \eqref{E54} yields
\begin{align*}
I_{i}^{1}=&\int_{\mathbb{R}^{m}}R\,[b_{i},\,D]\,{}^{t}R : \nry u^{\eps}(t) \otimes b_{i}\cdot\nabla_{y}(\nry u^{\eps}(t))\,\text{d}y \\
-& \int_{\mathbb{R}^{m}}R\,[b_{i},\,D]\,{}^{t}R : \nry u^{\eps}(t) \otimes \mathcal{A}_i \nry \ue (t)\,\text{d}y\\
=:&J_{i}^{1} + J_{i}^{2}.
\end{align*}
For estimating the term $J^1_i$, we use the symmetry of the matrix field $D$ and the formula
\begin{align*}
b_{i}\cdot\nabla_{y}\big(R\,[b_{i},\,D]\,{}^{t}R : \nry u^{\eps}(t)\otimes \nry u^{\eps}(t)\big) & = b_i \cdny ( R\,[b_{i},\,D]\,{}^{t}R ) : \nry u^{\eps}(t)\otimes \nry u^{\eps}(t) \\
& + 2R[b_i,D]\;^tR : \nry \ue (t) \otimes b_i \cdny (\nry \ue (t)).
\end{align*}
Integrating by parts leads to
\begin{align*}
2 \left |\sum _{i =1 } ^m J^1_i\right | & = \left | - \sum _{i =1 } ^m \inty{\divy b_i \,R\,[b_{i},\,D]\,{}^{t}R : \nry u^{\eps}(t)\otimes \nry u^{\eps}(t)} \right.\\
& \left. - \inty{\sum _{i = 1} ^m b_i \cdny (R\,[b_{i},\,D]\,{}^{t}R ): \nry u^{\eps}(t)\otimes \nry u^{\eps}(t)} \right |\\
& \leq \left ( \sum _{i = 1} ^m \|\divy b_i \|_{L^\infty} \|R[b_i,D]\,^tR\|_{L^\infty} + \left \|\sum _{i = 1} ^m b_i \cdny (R[b_i,D] \,^t R)\right \|_{L^\infty} \right )\\
& \times  \|\nry \ue (t) \|_{\lty} ^2.
\end{align*}
The estimate for the term $J_i ^2$ follows immediately, thanks to the hypothesis \eqref{E49}
\[
|J_i ^2| \leq \|R[b_i,D]\,^t R \|_{\liy} \|\mathcal{A}_i\|_{\liy} \|\nry \ue (t) \|_{\lty} ^2
\]
and finally there is a constant $C_1$ depending on $\max _{1\leq i \leq m} \|\divy b_i \|_{L^\infty}$, $\max _{1\leq i \leq m}\|\mathcal{A}_i \|_{L^\infty}$, $\max _{1\leq i \leq m}\|R [b_i,D]\,^t R\|_{L^\infty}$, $\|\sum _{i = 1} ^m b_i \cdny ( R [b_i, D]\,^t R)\|_{L^\infty}$ such that 
\begin{equation}
\label{I1}
\left | \sum _{i = 1} ^m I_i ^1 \right | \leq C_1 \|\nry \ue (t) \|_{\lty} ^2,\;\;t \in \R_+,\;\;\eps >0.
\end{equation}
For the term $I_i^2$ we can write, using the inequality $(R(y)D(y)\,^t R(y): R(y)D(y)\,^t R(y))^{1/2} \leq |D|_{H_Q^\infty}, y \in \R^m$
\begin{align*}
|I_i ^2| & = \left | \inty{RD \,^t R\,^t R^{-1} \nabla _y \ue (t) \cdot \,^t R^{-1} \nabla _y (\divy b_i)\, (b_i \cdny \ue (t))}\right|\\
& \leq |D|_{H_Q^\infty} \|\nry \divy b_i\|_{\liy} \inty{|\nry \ue (t)|\,|b_i \cdny \ue (t) |}\\
& \leq |D|_{H_Q^\infty} \|\nry \divy b_i\|_{\liy}\|\nry \ue (t)\|_{\lty} ^2
\end{align*}
implying that there is a constant $C_2$ depending on $|D|_{H_Q^\infty}, \max _{1\leq i \leq m}\|\nry \divy b_i\|_{\liy}$ such that 
\begin{equation}
\label{I2} 
\left | \sum _{i = 1} ^m I_i ^2 \right | \leq C_2 \|\nry \ue (t) \|_{\lty} ^2,\;\;t \in \R_+,\;\;\eps >0.
\end{equation}
Combining \eqref{2}, \eqref{I1}, \eqref{I2}  and applying Gronwall's lemma imply
\[
\|\nry \ue (t) \|_{\lty} \leq e ^{(C_1 + C_2)t}\|\nry \uin \|_{\lty},\;\;t \in \R_+,\;\;\eps >0
\]
and
\[
\|\nry \otimes \nry \ue \|_{L^2([0,T];\lty)} \leq \frac{e ^{(C_1 + C_2)t}}{\sqrt{2\alpha}}\|\nry \uin \|_{\lty},\;\;\eps >0.
\]
The first and second conclusions follow thanks to Remark \ref{Estveps}. For the last one, notice that 
\begin{align}
\label{DerTimeVeps} 
\partial _t \ve (t,z) & = \partial _t \ue (t, \ytz) + \frac{1}{\eps} b(\ytz) \cdny \ue (t, \ytz) \\
& = \divy(D\nabla _y \ue (t)) (\ytz) \nonumber
\end{align}
which implies 
\[
\|\partial _t \ve (t) \|_{\lty} = \|\divy ( D \nabla _y \ue (t))\|_{\lty}.
\]
By direct computation we obtain
\begin{align*}
\divy (D \nabla _y \ue ) & = \divy (D\,^t R \nry \ue ) = \divy(RD) \cdot \nry \ue + RD\,^t R : \partial _y (\nry \ue) R^{-1}\\
& = \divy (RD) \cdot \nry \ue + RD\,^t R : \nry \otimes \nry \ue
\end{align*}
and therefore
\begin{align*}
\supe \|\partial _t \ve \|_{L^2([0,T];\lty)} & = \supe \|\divy(D\nabla _y \ue ) \|_{L^2([0,T];\lty)} \\
& \leq \sqrt{T} \|\divy (RD)\|_{\liy} \supe \|\nry \ue \|_{L^\infty([0,T];\lty)} \\
& + |D|_{H_Q^\infty} \supe \|\nry \otimes \nry \ue\|_{L^2([0,T];\lty)},\;\;T \in \R_+.
\end{align*}
\end{proof}
Performing similar computations, we can propagate more regularity. The goal is to obtain a uniform bound for $(\partial _t \nrz \ve)_{\eps >0}$ in $L^2_{\mathrm{loc}}(\R_+;\lty)$. This can be achieved for any initial condition $\uin \in H^2 _R$. The proof is postponed to Appendix \ref{A}.
\begin{pro}
\label{MoreEstimatesBis}
Assume that the hypotheses \eqref{Equ15}, \eqref{Equ17}, \eqref{Equ58}, \eqref{Equ58Bis}, \eqref{E45}, \eqref{E46}, \eqref{E49} hold true. Moreover we assume that for any $i, j, k \in \{1,...,m\}$
\[
\nry \alpha _{ij} ^k \in \liy,\;\;\nry \divy b_j \in \liy,\;\;\nry \otimes \nry \divy b_j \in \liy,\;
\]
\[
\divy (R[b_i,D]) \in \liy,\;\;R[b_i,D]\,^t R \in \liy, \;\;\sum_{i=1}^m b_i \cdny (R[b_i,D]\,^t R ) \in \liy
\]
\[
R[b_j,[b_i,D]]\,^t R \in \liy,\;\;\divy (R[b_j,[b_i,D]]) \in \liy 
\]
\[
\nry (RD\,^t R) \in \liy,\;\;\nry \otimes \nry (RD\,^t R)\in \liy.
\]
If the initial condition $\uin$ belongs to $H^2_R$, then for any $T \in \R_+$ we have
\[
\supe \|\nry \otimes \nry \ue \|_{L^\infty([0,T];\lty)} = \supe \|\nrz \otimes \nrz \ve\|_{L^\infty([0,T];\lty)} < +\infty
\]
\[
\supe \|\nry \otimes \nry \otimes \nry \ue \|_{L^2([0,T];\lty)} = \supe \|\nrz \otimes \nrz \otimes \nrz \ve\|_{L^2([0,T];\lty)} < +\infty
\]
and
\[
\supe \|\partial _t \nrz \ve \|_{L^2([0,T];\lty)} < +\infty.
\]
Here the notation $\nabla ^R \otimes \nabla ^R \otimes \nabla ^R w$ stands for the tensor whose entry $(i,j,k)$ is $b_k \cdot \nabla (b_j \cdot \nabla (b_i \cdot \nabla w))$. 
\end{pro}
Similar computations allow us to estimate the solution of the limit model  \eqref{Equ14}. The arguments are a little bit tedious and we refer to Appendix \ref{A} for details.

\begin{pro}
\label{MoreEstimatesv}
Assume that all the hypotheses of Proposition \ref{MoreEstimates} hold true. Then we have for any $T \in \R_+$
\[
\nrz v\in L^\infty([0,T];\lty), \nrz \otimes \nrz v \in L^2([0,T];\lty),\;\;\partial _t v \in L^2([0,T];\lty).
\]
\end{pro}
\begin{pro}
\label{MoreEstimatesBisv}
Assume that all the hypotheses of Proposition \ref{MoreEstimatesBis} hold true. Then for any $T \in \R_+$, we have 
\[
\nrz \otimes \nrz v \in L^\infty([0,T];L^2),\;\;\nrz \otimes \nrz \otimes \nrz v \in L^2([0,T];L^2),\;\;\partial _t \nrz v \in L^2([0,T];L^2).
\]
\end{pro}
\begin{proof}
Apply exactly the same arguments as in the proof of Proposition \ref{MoreEstimatesBis}, after observing that the matrix field $\ave{D}$ satisfies the same hypotheses as the matrix field $D$ (see the proof of Proposition \ref{MoreEstimatesv}).
\end{proof}

\section{Two-scale analysis}
\label{2scale}
We intend to investigate the asymptotic behavior of \eqref{Equ1}, or equivalently \eqref{Equ8}. For any smooth, compactly supported function $\psi (t,z)$ we have to pass to the limit, when $\eps \searrow 0$, in the formulation
\[
- \intz{\uin (z) \psi (0,z)} - \inttz{\ve (t,z) \partial _t \psi }+ \inttz{G(t/\eps)D\nabla _z \ve \cdot \nabla _z \psi } = 0.
\]
Clearly, the main difficulty comes from the last integral, which presents two time scales ~: a slow time variable $t$ and also a fast time variable $s = t/\eps$ (not necessarily periodic). We detail here a general two-scale convergence result, based on ergodic means. Let us introduce some notations. We denote by $\avepq{\cdot}{\cdot}:H_P \times H_Q \to \R$ the bilinear continuous application defined for any $(A, B) \in H_P \times H_Q$ by 
\[
\avepq{A}{B}  = \inty{A(y):B(y)} = \inty{P^{1/2}AP^{1/2}:Q^{1/2}BQ^{1/2}}\leq |A|_P|B|_Q.
\]
It is easily seen that $A\in H_P \to \avepq{A}{\cdot}\in H_Q^\prime$ is a linear isomorphism, $\|\avepq{A}{\cdot}\|_{H_Q^\prime} = |A|_P$. Therefore we identify $H_Q^\prime$ to $H_P$ through the duality $\avepq{\cdot}{\cdot}$. Notice also that $A\in H_P \to PAP \in H_Q$, $B \in H_Q \to QBQ \in H_P$ are linear isomorphisms, $|PAP|_Q = |A|_P,|QBQ|_P = |B|_Q$. 
\begin{pro}
\label{SlowFastMat}
Let $T$ be a positive real number. Consider $C \in L^\infty(\R;H_Q)$, such that the family of means $\left ( \frac{1}{S}\int _{s_0} ^{s_0 +S } C(s)\;\md s \right ) _{S>0}$ converges strongly in $H_Q$ toward some $\overline{C} \in H_Q$, uniformly with respect to $s_0 \in \R$, when $S \to +\infty$ and $\calB _\omega \subset C([0,T];H_P)$ a bounded set in $L^1([0,T];H_P)$, of functions which admit as modulus of continuity in $C([0,T];H_P)$ the same function $\omega : [0,T]\to \R_+$ {\it i.e.,}
\[
|B(t) - B(t^\prime)|_P \leq \omega ( |t - t^\prime|),\;\;t, t^\prime \in [0,T],\;\;B \in \calB _\omega
\]
with $\omega$ non decreasing and $\lim _{\lambda \searrow 0} \omega (\lambda)= 0$. Then 
\[
\lime \int _0 ^T \avepq{B(t)}{C(t/\eps)}\;\md t = \int _0 ^T \avepq{B(t)}{\overline{C}}\;\md t
\]
uniformly with respect to $B \in \calB _\omega$. 
\end{pro}
\begin{proof}
For any $\delta >0$, there is $S_\delta >0$ such that 
\[
\left | \frac{1}{S} \int _{s_0} ^{s_0 + S}C(s) \;\md s - \overline{C}\right |_Q < \delta,\;\;\mbox{for any } S\geq S_\delta \mbox{ and } s_0 \in \R.
\]
Performing the change of variable $ s = \frac{t}{\eps}$ in the above integral, leads to
\begin{equation}
\label{E82}
\left | \frac{1}{T}\int_{t_0} ^{t_0 + T} C(t/\eps) \;\md t - \overline{C} \right |_Q < \delta,\;\;\mbox{for any } T\geq \eps S_\delta = T_{\delta, \eps}\;\mbox{ and } t_0 \in \R.
\end{equation}
We split the interval $[0,T[$ in a finite number of intervals of size great or equal to $\Tde$. For example let $\kde$ be $\left [ \frac{T}{\Tde} \right ]$. If $T/\Tde$ is an integer, that is $T/\Tde = \kde$, we consider the intervals
\[
[k\Tde, (k+1) \Tde [,\;\;0 \leq k \leq \kde - 1
\]
and if $T/\Tde$ is not an integer, we take the intervals
\[
[k\Tde, (k+1) \Tde [,\;\;0 \leq k \leq  \kde- 2,\;\;\mbox{ and } \;[(\kde - 1) \Tde, T[.
\]
Notice that in both cases we have $\kde$ intervals, whose sizes are between $\Tde$ and $2\Tde$. We denote by $(t_{k, \delta, \eps})_{0\leq k \leq \kde}$, or simply $(t_k)_{0 \leq k \leq \kde}$, the end points of these intervals. The last point is allways $t_{\kde} = T$. Therefore we can write for any $B \in \calB _\omega$
\begin{align}
\label{E84}
\left | 
\inttT{\avepq{B(t)}{C(t/\eps)}}  \right. & -\left . \inttT{\avepq{B(t)}{\overline{C}}} \right |  = 
\left | \inttT{\avepq{B(t)}{C(t/\eps) - \overline{C}}}\right | \nonumber\\
& \leq \sum _{k = 0} ^{k_{\delta, \eps}-1} \left |\int_{t_k} ^{t_{k+1}} \avepq{B(t)}{C(t/\eps) - \overline{C}}\md t  \right |\nonumber \\
& \leq \sum _{k = 0} ^{k_{\delta, \eps} - 1}\left |\int_{t_k} ^{t_{k+1}} \avepq{B(t)- B(t_k)}{C(t/\eps) - \overline{C}}\md t  \right |\nonumber \\
& + \sum _{k = 0} ^{k_{\delta, \eps}-1} \left | \int _{t_k} ^{t_{k+1}} \avepq{B(t_k)}{C(t/\eps) - \overline{C}}\;\md t\right|\nonumber \\
& =: \Sigma _1 + \Sigma _2.
\end{align}
Since the function $t \in [0,T]\to B(t) \in H_P$ admits $\omega$ as modulus of continuity, we obtain the following estimate for $\Sigma _1$
\begin{align}
\label{E81}
\Sigma _1 & \leq \sum _{k = 0} ^{k_{\delta, \eps}-1} \int _{t_k } ^{t_{k+1} }\omega ( |t - t_k |)\;|C(t/\eps) - \overline{C}|_Q \;\md t \\
& \leq \sum _{k = 0} ^ {k_{\delta, \eps}-1}\omega (2T_{\delta, \eps})(t_{k+1} -t_k) \;2 \|C\|_{L^\infty(\R; H_Q)} \nonumber \\
& = 2 \|C\|_{L^\infty(\R; H_Q)} \omega (2T_{\delta, \eps})T.\nonumber
\end{align}
The estimate for $\Sigma_2$ comes by using \eqref{E82}
\begin{align}
\label{E83}
\Sigma_2 & = \sum _{k = 0} ^{k_{\delta, \eps}-1}\left |\int _{t_k } ^{t_{k+1} }(PB(t_k)P,C(t/\eps) - \overline{C})_Q\;\md t  \right |\\
& = \sum _{k = 0} ^{\kde -1 } \left |
\left ( PB (t_k)P,\int _{t_k} ^ {t_{k+1}} (C(t/\eps) - \overline{C})\;\md t\right )_Q  \right| \nonumber \\
& = \sum _{k = 0} ^{\kde - 1} \left | \avepq{B (t_k )}{\int _{t_k } ^ {t_{k+1}}
(C(t/\eps) - \overline{C})\;\md t} \right |\nonumber \\
& \leq \sum _{k = 0} ^{\kde - 1}\delta (t_{k+1} - t_k)  |B(t_k)|_P \nonumber \\
& \leq \delta \left [\|B\|_{L^1([0,T];H_P)} + \omega (2\Tde) T \right ].\nonumber 
\end{align}
Thanks to \eqref{E84}, \eqref{E81}, \eqref{E83} we deduce
\begin{align*}
\left |\inttT{\avepq{B(t)}{C(t/\eps)}}  - \inttT{\avepq{B(t)}{\overline{C}}}\right| & \leq 2 \|C\|_{L^\infty(\R; H_Q)} \omega (2\Tde) T \\
& + \delta \left [\|B\|_{L^1([0,T];H_P)} + \omega (2\Tde) T    \right ].
\end{align*}
Let $\eta$ be a positive real number and $\delta >0$ small enough such that $\delta \|B\|_{L^1([0,T];H_P)}< \eta /2$ uniformly with respect to $B\in \calB_\omega$ (which is possible since $\calB_\omega$ is bounded in $L^1([0,T];H_P)$). Observing that $\lime \Tde = \lime \eps S_\delta = 0$, and $\lime \omega (2\Tde) = 0$, we deduce that there is $\eps = \eps (\eta)$ such that for any $0 < \eps < \eps (\eta)$
\[
2 \|C\|_{L^\infty(\R; H_Q)} \omega (2\Tde) T  + \delta  \omega (2\Tde) T   < \frac{\eta}{2}.
\]
Finally we obtain
\begin{align*}
\left | \inttT{\avepq{B(t)}{C(t/\eps)}}    - \inttT{\avepq{B(t)}{\overline{C}}}\right| & \leq \delta \|B\|_{L^1([0,T]; H_P)} + \frac{\eta}{2} < \eta
\end{align*}
for any $0 < \eps < \eps (\eta)$, uniformly with respect to $B \in \calB_\omega$. 
\end{proof}
\begin{remark}
\label{SlowFastMatBis}
The conclusion of Proposition \ref{SlowFastMat} holds true for any pair $(B,C) \in L^1([0,T];H_P)\times L^\infty(\R;H_Q)$ such that $\left ( \frac{1}{S}\int _{s_0} ^{s_0 + S}C(s) \;\md s\right )_{S>0}$ converges strongly in $H_Q$ toward some $\overline{C} \in H_Q$, uniformly with respect to $s_0 \in \R$, when $S \to +\infty$. Indeed, observe that 
\[
\left | \int _0 ^T \avepq{B(t)}{C(t/\eps)}\;\md t -  \int _0 ^T \avepq{B(t)}{\overline{C}}\;\md t \right | \leq 2 \|B\|_{L^1([0,T];H_P)} \|C\|_{L^\infty(\R;H_Q)}
\]
and thus, by using the density of $C([0,T];H_P)$ in $L^1([0,T];H_P)$, it is enough to consider $B \in C([0,T];H_P)$. But in this case, the uniform continuity of $B$ allows us to pick a modulus of continuity $\omega : [0,T]\to \R_+$
\[
\omega (\lambda ) = \sup _{t, t^\prime \in [0,T],|t- t^\prime| \leq \lambda} |B(t) - B(t^\prime)|_P,\;\;\lambda \in [0,T]
\]
and all the arguments in the proof of Proposition \ref{SlowFastMat} apply.
\end{remark}
In the sequel, we present some consequences of Proposition \ref{SlowFastMat} which will be used when justifying the main result in Theorem \ref{MainRes1}.
\begin{pro}
\label{WeakLim}
Let $T$ be a positive real number. Consider $D \in H_Q \cap H_Q^\infty$ a symmetric matrix field and $\mathcal{W}_\omega \subset C([0,T];X_P)$ a bounded set in $L^2([0,T];X_P)$ of functions which admit as modulus of continuity in $C([0,T];X_P)$ the same function $\omega : [0,T]\to \R_+$, {\it i.e.,}
\[
|w(t) - w(t^\prime)|_P \leq \omega (|t - t ^\prime|),\;\;t, t^\prime \in [0,T],\;\;w \in \mathcal{W}_\omega
\]
with $\omega$ non decreasing and $\lim _{\lambda \searrow 0} \omega (\lambda) = 0$. Then for any family $(w ^\beta)_{\beta >0} \subset \mathcal{W}_\omega$ which converges weakly in $L^2([0,T];X_P)$ toward $w^0$ when $\beta \searrow 0$, we have
\begin{equation}
\label{E86}
\lim _{(\beta, \eps) \to (0,0)} \int _0 ^T \avepq{\theta (t) \otimes w^\beta (t)}{G(t/\eps)D}\;\md t = \int _0 ^T \avepq{\theta (t) \otimes w^0 (t)}{\ave{D}}\;\md t
\end{equation}
for any $\theta \in L^2([0,T];X_P)$. 
\end{pro}
\begin{proof}
Notice that for any $\theta, w \in L^2([0,T];X_P)$ we have 
\begin{align*}
\int _0 ^T & \avepq{\theta (t)  \otimes w (t)}{G(t/\eps)D}\;\md t  = \inttTy{\theta(t,y) \otimes w(t,y) : G(t/\eps)D} \\
& = \inttTy{( P^{1/2}(y)\theta(t,y) )\otimes (P^{1/2}(y)w(t,y)) : Q^{1/2}(y)G(t/\eps)D Q^{1/2}(y)} \\
& \leq \int _0 ^T |G(t/\eps)D|_{H_Q^\infty}\inty{|P^{1/2}(y)\theta(t,y)|\,|P^{1/2}(y)w(t,y)|}\md t \\
& \leq |D|_{H_Q^\infty}\left (\inttTy{P(y) \theta (t,y) \cdot \theta(t,y)} \right ) ^{1/2}\left (\inttTy{P(y) w (t,y) \cdot w(t,y)} \right ) ^{1/2}\\
& = |D|_{H_Q^\infty} \|\theta \|_{L^2([0,T];X_P)} \|w \|_{L^2([0,T];X_P)} 
\end{align*}
and similarly, by using $|\ave{D}|_{H_Q^\infty} \leq |D|_{H_Q^\infty}$
\begin{equation}
\label{E87}
\int _0 ^T  \avepq{\theta (t)  \otimes w (t)}{\ave{D}}\;\md t \leq |D|_{H_Q^\infty} \|\theta \|_{L^2([0,T];X_P)} \|w \|_{L^2([0,T];X_P)}.
\end{equation}
As the family $(w^\beta)_{\beta >0}$ is bounded in $L^2([0,T];X_P)$, it is enough to check \eqref{E86} for any $\theta$ in a dense subset of $L^2([0,T];X_P)$, for example for any $\theta$ such that $P^{1/2} \theta \in C^0 _c ([0,T]\times \R^m)$. We appeal to Proposition \ref{SlowFastMat} with $C(s) = G(s)D, \overline{C} = \ave{D}$ and $\calB = \{\theta \otimes w\;:\; w \in \mathcal{W}_\omega\}$. By Proposition \ref{CZGroup} we know that $(G(s))_{s\in \R}$ is a $C^0$-group of unitary operators on $H_Q$, implying that $C \in L^\infty(\R;H_Q)$. By Theorem \ref{AveMatField} we deduce that
\[
\limS \frac{1}{S} \int _{s_0} ^{s_0 + S} C(s)\;\md s = \overline{C},\;\mbox{uniformly with respect to } s_0 \in \R.
\]
For any $w \in \mathcal{W}_\omega$ we write
\begin{align*}
\|\theta \otimes w\|_{L^1([0,T];H_P)} & = \int_0 ^T \left (\inty{(P^{1/2} \theta) \otimes (P^{1/2} w): (P^{1/2} \theta) \otimes (P^{1/2} w)} \right ) ^{1/2} \;\md t\\
& \leq \int _0 ^T |\theta (t) |_{X_P ^\infty} |w(t)|_P\;\md t\\
& \leq \|P^{1/2} \theta\|_{L^2([0,T];\liy)}\|w \|_{L^2([0,T];X_P)}
\end{align*}
and therefore the boundedness of $\mathcal{W}_\omega$ in $L^2([0,T];X_P)$ implies the boundedness of $\calB$ in $L^1([0,T];H_P)$ (here use $P^{1/2} \theta \in C^0 _c([0,T]\times \R^m)$). We search now for a continuity modulus of $\calB$. For any $w \in \mathcal{W}_\omega$, $t, t^\prime \in [0,T]$, we have
\begin{align*}
|\theta (t) \otimes w(t) & - \theta (t^\prime) \otimes w(t^\prime)|_P  \leq |\theta (t) - \theta (t^\prime) |_{X_P^\infty} |w(t)|_P + |\theta(t^\prime)|_{X_P^\infty} |w(t) - w(t^\prime)|_P \\
& \leq \|P^{1/2} \theta (t) - P^{1/2} \theta (t^\prime) \|_{\liy} |w(t)|_P + \|P^{1/2}\theta(t^\prime)\|_{\liy} \omega( |t - t^\prime|) \\
& \leq \omega_\theta ( |t - t^\prime|)\|w\|_{C([0,T];X_P)} + \omega( |t - t^\prime|)\|P^{1/2}\theta\|_{C^0([0,T]\times \R^m)}
\end{align*}
where $\omega_\theta$ is a continuity modulus for $P^{1/2}\theta \in C^0_c([0,T]\times \R^m)$. We are done if we show that $\mathcal{W}_\omega$ is also bounded in $C([0,T];X_P)$. This comes easily by noticing that for any $t, t^\prime \in [0,T], w \in \mathcal{W}_\omega$ we have
\[
|w(t)|^2_P \leq ( |w(t^\prime)| + \omega (|t-t^\prime|))^2 \leq 2 |w(t^\prime)|^2 _P + 2 \omega ^2 (T).
\]
Integrating with respect to $t^\prime \in [0,T]$ one gets for any $t \in [0,T]$
\[
|w(t)|^2_P \leq \frac{2}{T} \|w\|^2 _{L^2([0,T];X_P)} + 2 \omega ^2(T) 
\]
saying that $\mathcal{W}_\omega$ is bounded in $C([0,T];X_P)$. By Proposition \ref{SlowFastMat}, for any $\eta >0$, there is $\eps (\eta)>0$ such that for any $0 < \eps < \eps (\eta), \beta >0$
\begin{equation*}
\label{E88}
\left | \int _0 ^T \avepq{\theta(t) \otimes w^\beta (t)}{G(t/\eps)D}\;\md t - \int _0 ^T \avepq{\theta(t) \otimes w^\beta (t)}{\ave{D}}\;\md t \right | < \frac{\eta}{2}
\end{equation*}
By \eqref{E87} we know that $w \to \int_0 ^T \avepq{\theta(t)\otimes w(t)}{\ave{D}}\;\md t$ is a linear continuous application on $L^2([0,T];X_P)$, and since $(w^\beta)_{\beta >0}$ converges weakly in $L^2([0,T];X_P)$, toward $w^0$, when $\beta \searrow 0$, there is $\beta (\eta) >0$ such that for any $0 < \beta < \beta (\eta)$
\begin{equation*}
\label{E89}
\left | \int_0 ^T \avepq{\theta(t)\otimes w^\beta (t)}{\ave{D}}\;\md t - \int _0 ^T \avepq{\theta(t) \otimes w^0}{\ave{D}}\;\md t \right | < \frac{\eta}{2}
\end{equation*}
Therefore, for any $\eta >0$, there is $\beta (\eta) >0, \eps (\eta) >0$ such that for any $0 < \beta < \beta (\eta),0< \eps < \eps (\eta)$
\[
\left | \int_0 ^T \avepq{\theta(t)\otimes w^\beta (t)}{G(t/\eps)D}\;\md t - \int _0 ^T \avepq{\theta(t) \otimes w^0}{\ave{D}}\;\md t \right | < \eta.
\]
\end{proof}
\begin{remark}
\label{Aux} The previous arguments show that if $D \in H_Q \cap H_Q^\infty$, then
\begin{equation}
\label{E94}
\lime \int _0 ^T \avepq{w(t) \otimes w(t)}{G(t/\eps)D}\;\md t = \int _0 ^T \avepq{w(t) \otimes w(t)}{\ave{D}}\;\md t
\end{equation}
for any $w \in L^2([0,T];X_P)$. Indeed, taking into account that the bilinear application 
\[
(\theta, w) \in L^2([0,T];X_P) \times L^2([0,T];X_P) \to \int _0 ^T \avepq{\theta(t) \otimes w(t)}{\ave{D}}\;\md t \in \R
\]
is continuous, it is enough to establish \eqref{E94} for $w$ in the set $\{\theta \in L^2([0,T];X_P)\;:\;P^{1/2} \theta \in C^0_c([0,T]\times \R^m)\}$, which is dense in $L^2([0,T];X_P)$. And this is a direct consequence of Remark \ref{SlowFastMatBis}, since for any $\theta \in L^2([0,T];X_P)$ such that $P^{1/2}\theta \in C^0_c([0,T]\times \R^m)$, we have
\begin{align*}
\|\theta \otimes \theta \|_{L^1([0,T];H_P)} & = \int _0 ^T \!\!\!\left (\inty{|P^{1/2}\theta|^4} \right )^{1/2}\!\!\!\md t \leq \!\!\int _0 ^T \!\!\!\|P^{1/2}\theta (t)\|_{C^0(\R^m)} \|P^{1/2}\theta (t)\|_{\lty}\;\md t \\
& \leq \|\theta\|_{L^2([0,T];X_P^\infty)} \|\theta\|_{L^2([0,T];X_P)} < +\infty.
\end{align*}
\end{remark}
When the matrix field $D$ is definite positive, the behavior of the upper limit with respect to $(\beta, \eps)$ for the quadratic term $\int_0 ^T \avepq{\theta(t)\otimes w^\beta (t)}{G(t/\eps)D}\;\md t$ characterizes the strong convergence of the family $(w^\beta)_{\beta >0}$ as shown in the following result.
\begin{pro}
\label{StrongLim}
Assume that all the hypotheses in Proposition \ref{WeakLim} hold true.
\begin{enumerate}
\item
If the matrix field $D$ is positive, then we have
\[
\int _0 ^T \avepq{w^0(t) \otimes w^0(t)}{\ave{D}}\;\md t \leq \liminf _{(\beta, \eps) \to (0,0)} \int _0 ^T \avepq{w^\beta(t) \otimes w^\beta(t)}{G(t/\eps)D}\;\md t. 
\]
\item
If $(w^\beta)_{\beta>0}$ converges strongly in $L^2([0,T];X_P)$ toward $w^0$ when $\beta \searrow 0$ (the existence of a modulus of continuity $\omega$ in $C([0,T];X_P)$ for the family $(w^\beta)_{\beta >0}$ is not necessary here), then we have
\[
\int _0 ^T \avepq{w^0(t) \otimes w^0(t)}{\ave{D}}\;\md t = \lim _{(\beta, \eps) \to (0,0)} \int _0 ^T \avepq{w^\beta(t) \otimes w^\beta(t)}{G(t/\eps)D}\;\md t. 
\]
\item
If there is $\alpha >0$ such that $Q^{1/2} D Q^{1/2} \geq \alpha I_m$, and
\[
\limsup _{(\beta, \eps) \to (0,0)} \int _0 ^T \avepq{w^\beta(t) \otimes w^\beta(t)}{G(t/\eps)D}\;\md t \leq  \int _0 ^T \avepq{w^0(t) \otimes w^0(t)}{\ave{D}}\;\md t
\]
then the family $(w^\beta)_{\beta >0}$ converges strongly in $L^2([0,T];X_P)$ toward $w^0$ when $\beta \searrow 0$. 
\end{enumerate}
\end{pro}
\begin{proof} $\;$\\
1. As the matrix field $D$ is symmetric and positive, so is the matrix field $G(t/\eps)D$ for any $t \in [0,T]$ and $\eps >0$, and thus
\begin{align*}
\int_0 ^T & \avepq{w^0(t) \otimes w^\beta(t)}{G(t/\eps)D}\;\md t  = \inttTy{ w^0(t,y) \otimes w^\beta(t,y):G(t/\eps)D} \\
&  \leq \inttTy{(w^0(t,y) \otimes w^0(t,y):G(t/\eps)D)^{1/2}\; (w^\beta(t,y) \otimes w^\beta(t,y):G(t/\eps)D)^{1/2} } \\
& \leq \left ( \int _0 ^T \!\!\!\!\avepq{w^0(t) \otimes w^0(t)}{G(t/\eps)D}\;\md t  \right ) ^{1/2} \!\! \left (\int _0 ^T \!\!\!\!\avepq{w^\beta(t) \otimes w^\beta(t)}{G(t/\eps)D}\;\md t  \right ) ^{1/2}.
\end{align*}
Passing to the lower limit with respect to $(\beta, \eps)$ yields, thanks to Proposition \ref{WeakLim}
\begin{align}
\label{E91}
\int _0 ^T  \avepq{w^0(t)\otimes w^0(t)}{\ave{D}}\;\md t & \leq 
\liminf _{(\beta, \eps) \to (0,0)} \left \{ \left ( \int _0 ^T \avepq{w^0(t) \otimes w^0(t)}{G(t/\eps)D}\;\md t  \right ) ^{1/2} \right. \nonumber \\
& \times \left. \left (\int _0 ^T \avepq{w^\beta(t) \otimes w^\beta(t)}{G(t/\eps)D}\;\md t  \right ) ^{1/2} \right \}.
\end{align}
Thanks to Remark \ref{Aux}, we know that 
\begin{equation}
\label{E92}
\lime \int _0 ^T  \avepq{w^0(t)\otimes w^0(t)}{G(t/\eps)D}\;\md t = \int _0 ^T  \avepq{w^0(t)\otimes w^0(t)}{\ave{D}}\;\md t.
\end{equation}
Using the equality \eqref{E92} in the inequality \eqref{E91} leads to
\begin{align*}
\int _0 ^T  \avepq{w^0(t)\otimes w^0(t)}{\ave{D}}\;\md t &\leq \left (\int _0 ^T  \avepq{w^0(t)\otimes w^0(t)}{\ave{D}}\;\md t \right )^{1/2} \\
& \times \liminf_{(\beta, \eps) \to (0,0)} \left ( \int _0 ^T  \avepq{w^\beta(t)\otimes w^\beta(t)}{G(t/\eps)D}\;\md t\right )^{1/2}
\end{align*}
which is equivalent to our assertion.\\
2. Pick $\eta$ a positive real number. By Remark \ref{Aux}, there is $\eps (\eta)$ such that for any $0 < \eps < \eps (\eta)$
\[
\left | \int _0 ^T \avepq{w^0(t)\otimes w^0(t)}{G(t/\eps)D}\;\md t - \int_0 ^T \avepq{w^0(t)\otimes w^0(t)}{\ave{D}}\;\md t\right | < \frac{\eta}{2}.
\]
It is easily seen, thanks to the strong convergence of $(w^\beta)_{\beta>0}$ in $L^2([0,T];X_P)$ toward $w^0$, that there is $\beta (\eta)>0$ such that for any $0 < \beta < \beta(\eta), \eps >0$
\begin{align*}
& \left | \int _0 ^T \avepq{w^\beta(t)\otimes w^\beta(t)}{G(t/\eps)D}\;\md t - \int _0 ^T \avepq{w^0(t)\otimes w^0(t)}{G(t/\eps)D}\;\md t \right |  \\
& \leq |D|_{H_Q^\infty} \|w^\beta - w^0\|_{L^2([0,T];X_P)}\left ( \|w^\beta \|_{L^2([0,T];X_P)} + \|w^0 \|_{L^2([0,T];X_P)}\right) < \frac{\eta}{2}.
\end{align*}
Therefore the second assertion holds true, that is, for any $\eta >0$, there is $\beta (\eta) >0, \eps(\eta)>0$ such that 
\[
\left |  \int _0 ^T \avepq{w^\beta(t)\otimes w^\beta(t)}{G(t/\eps)D}\;\md t -  \int _0 ^T \avepq{w^0(t)\otimes w^0(t)}{\ave{D}}\;\md t\right | < \eta
\]
for any $0 < \beta < \beta(\eta), 0 < \eps < \eps(\eta)$. \\
3. We know by Proposition \ref{CZGroup} that $Q^{1/2} G(t/\eps) Q^{1/2} \geq \alpha I_m$, for any $t \in \R_+, \eps >0$ and therefore 
\begin{align*}
\alpha \|w^\beta & - w^0 \|^2 _{L^2([0,T];X_P)}  \leq \int _0 ^T \avepq{[w^\beta (t) - w^0 (t)]\otimes [w^\beta (t) - w^0 (t)]}{G(t/\eps)D} \;\md t \\
& = \int _0 ^T \avepq{w^\beta (t) \otimes w^\beta (t)}{G(t/\eps)D} \;\md t + 
\int _0 ^T \avepq{w^0 (t) \otimes w^0 (t)}{G(t/\eps)D} \;\md t\\
& - \int _0 ^T \avepq{w^\beta (t) \otimes w^0 (t)}{G(t/\eps)D} \;\md t - \int _0 ^T \avepq{w^0 (t) \otimes w^\beta (t)}{G(t/\eps)D} \;\md t.
\end{align*}
By Proposition \ref{WeakLim} we know that 
\begin{align*}
\lim_{(\beta, \eps) \to (0,0)}\int _0 ^T \!\!\!\!\!\avepq{w^\beta (t) \otimes w^0 (t)}{G(t/\eps)D} \;\md t & = \lim_{(\beta, \eps) \to (0,0)}\int _0 ^T \!\!\!\!\!\avepq{w^0 (t) \otimes w^\beta (t)}{G(t/\eps)D} \;\md t \\
& = \int _0 ^T \avepq{w^0 (t) \otimes w^0 (t)}{\ave{D}} \;\md t 
\end{align*}
and by Remark \ref{Aux} we have
\[
\lime \int _0 ^T \avepq{w^0 (t) \otimes w^0 (t)}{G(t/\eps)D} \;\md t = \int _0 ^T \avepq{w^0 (t) \otimes w^0 (t)}{\ave{D}} \;\md t.
\]
Finally we obtain
\begin{align*}
\alpha \limsup _{\beta \searrow 0} \|w^\beta - w^0\|_{L^2([0,T];X_P)}^2 & \leq \limsup _{(\beta, \eps) \to (0,0)}\int _0 ^T \avepq{w^\beta (t) \otimes w^\beta (t)}{G(t/\eps)D} \;\md t \\
& - \int _0 ^T \avepq{w^0 (t) \otimes w^0 (t)}{\ave{D}} \;\md t \leq 0
\end{align*}
saying that $(w^\beta)_{\beta >0}$ converges strongly in $L^2([0,T];X_P)$ toward $w^0$ when $\beta \searrow 0$. 
\end{proof}

\section{Proofs of the main theorems}
\label{Proofs}

We establish two convergence results. In Theorem \ref{MainRes1} we prove strong convergence results for the families $(\ve)_{\eps>0}$ in $\litloclty$ and $(\nabla _z \ve )_{\eps>0}$ in $\lttlocxp$. In Theorem \ref{MainRes2} we study the order of the above convergences, by introducing a corrector, that is, we justify the dominant term in the developement \eqref{Equ10}.

\begin{proof} (of Theorem \ref{MainRes1})\\
As $\ue$ is the variational solution of \eqref{Equ1}, we have for any $\Phi \in C^1_c (\R_+ \times \R^m)$
\begin{align}
\label{E95} 
- \intty{\ue (t,y) \partial _t \Phi } & - \inty{\uin (y) \Phi (0,y)} + \intty{D(y) \nabla _y \ue \cdny \Phi} \nonumber \\
& - \frac{1}{\eps} \intty{\ue (t,y) b(y) \cdny \Phi } = 0.
\end{align}
Actually the above formulation holds true for any compactly supported function in $\R_+ \times \R^m$, which belongs to $W^{1,\infty}(\R_+ \times \R^m)$. Pick a test function $\psi \in C^1_c (\R_+ \times \R^m)$ and let us introduce the function $\Phie (t,y) = \psi (t, \ymty)$, $(t,y) \in \R_+ \times \R^m$. Thanks to the hypotheses \eqref{Equ15}, \eqref{Equ17}, the function $\Phie$ is compactly supported in $\R_+ \times \R^m$, belongs to $W^{1,\infty}(\R_+ \times \R^m)$ and thus satisfies \eqref{E95}. We perform the change of variable $y = \ytz$. Taking the time and space derivatives of the equalities $\psi (t,z) = \Phie (t,\ytz)$ and $\ve (t,z) = \ue (t, \ytz)$ gives
\[
\partial _t \psi (t,z)  = \partial _t \Phie (t,\ytz) + \frac{1}{\eps} b(\ytz) \cdny \Phie (t, \ytz)
\]
\[
\nabla _z \psi (t,z) = \,^t \partial \ytz \nabla _y \Phie (t,\ytz),\;\;\nabla _z \ve (t,z) = \,^t \partial \ytz \nabla _y \ue (t,\ytz)
\]
and the weak formulation \eqref{E95}, written with the test function $\Phie (t,y)$ becomes 
\begin{align*}
- \inttz{& \ve (t,z) \partial _t \psi } - \intz{\uin (z) \psi (0,z)} \\
& + \inttz{\partial Y^{-1}(t/\eps;z)D(\ytz)\,^t \partial Y^{-1}(t/\eps;z)\nabla _z \ve \cdot \nabla _z \psi } = 0.
\end{align*}
Therefore $\ve$ is the variational solution of \eqref{Equ8}. By Propositions \ref{MoreEstimates}, \ref{MoreEstimatesBis} we have, for any $T \in \R_+$
\[
\supe \{ \|\ve \|_{\litlty} + \|\nrz \ve \|_{L^\infty([0,T];\lty)} + \|\nrz \otimes \nrz \ve \|_{L^\infty([0,T];\lty)} \} < +\infty
\]
\[
\supe \{ \|\partial _t \ve \|_{L^2([0,T];\lty)} + \|\partial _t \nrz \ve \|_{L^2([0,T];\lty)}\} < +\infty.
\]
Let us consider a sequence $(\epsk)_k$ converging to $0$ such that 
\begin{equation}
\label{E101} \limk \vek = v^0\;\mbox{ weakly } \star \mbox{ in } \litlty
\end{equation}
\begin{equation}
\label{E102} \limk \nabla _z \vek = \nabla _z v^0\;\mbox{ weakly } \star \mbox{ in } L^\infty([0,T];X_P),\;\;T \in \R_+.
\end{equation}
We claim that $v^0$ is the variational solution of \eqref{Equ14}. For any $\eta \in C^1 _c (\R_+)$ and $\Phi \in H^1_R$, the variational formulation of \eqref{Equ8} yields
\begin{align*}
- \inttz{\vek (t,z) \eta ^\prime (t) \Phi (z)} & - \intz{\uin (z) \eta (0) \Phi (z)} \\
& + \inttz{G(t/\epsk)D\nabla _z \vek \cdot \eta (t) \nabla _z \Phi } = 0.
\end{align*}
As $\eta ^\prime \Phi$ belongs to $\lotlty$, the weak $\star$ convergence in $\litlty$ of $(\vek)_k$ gives
\[
\inttz{\vek (t,z) \eta ^\prime (t) \Phi (z)} \underset{k\to +\infty}{\longrightarrow} \inttz{v^0 (t,z) \eta ^\prime (t) \Phi (z)}.
\]
We use now Proposition \ref{WeakLim} with $T>0$ such that $\supp \eta \subset [0,T[$, and $\mathcal{W}_\omega = \{w^k = \nabla _z \vek|_{[0,T]\times \R^m} \;:\;k \in \N\}$. Obviously, $\mathcal{W}_\omega$ is bounded in $L^2([0,T];X_P)$ and for any $k\in \N$, $t, t^\prime \in [0,T]$, we can write
\[
|\nabla _z \vek (t) - \nabla _z \vek (t^\prime) |_P = \|\nrz \vek (t) - \nrz \vek (t^\prime)\|_{L^2} \leq \!\sqrt{|t - t^\prime|}\, \|\partial _t \nrz \vek \|_{L^2([0,T];\lty)}.
\]
Therefore $\mathcal{W}_\omega$ is contained in $C([0,T];X_P)$ and admits the continuity modulus
\[
\omega (\lambda) = \sqrt{\lambda} \supe \|\partial _t \nrz \ve \|_{L^2([0,T];\lty)}.
\]
Applying Proposition \ref{WeakLim} with $\theta(t,z) = \eta(t) \nabla _z \Phi (z) \in L^2([0,T];X_P)$ we deduce that 
\begin{align*}
 \inttz{G(t/\epsk)D\nabla _z \vek \cdot \eta (t) \nabla _z \Phi } & =  \int _0 ^T \avepq{\eta (t) \nabla _z \Phi \otimes \nabla _z \vek (t)}{G(t/\epsk)D}\;\md t \\
& \underset{k\to +\infty}{\longrightarrow} \int _0 ^T \avepq{\eta (t) \nabla _z \Phi \otimes \nabla _z v^0 (t)}{\ave{D}}\;\md t \\
& = \inttz{\ave{D}\nabla _z v^0 \cdot \eta (t) \nabla _z \Phi}.
\end{align*}
Therefore, passing to the limit, when $k \to +\infty$, in the variational formulation of $\vek$, implies
\begin{align*}
- \inttz{v^0 (t,z) \eta ^\prime (t) \Phi (z)} & - \intz{\uin (z) \eta (0) \Phi (z)} \\
& + \inttz{\ave{D}\nabla _z v^0 \cdot \eta (t) \nabla _z \Phi } = 0
\end{align*}
and thus $v^0$ is the variational solution of \eqref{Equ14} ($v^0 = v$). By the uniqueness of the solution for the limit model \eqref{Equ14}, we deduce that the convergences in \eqref{E101}, \eqref{E102} hold with respect to $\eps \searrow 0$ 
\[
\lime \ve = v \mbox{ weakly } \star \mbox{ in } \litlty,\;\;\lime \nabla _z \ve = \nabla _z v \mbox{ weakly } \star \mbox{ in } L^\infty_{\mathrm{loc}}(\R_+;X_P).
\] 
The regularity of $v$ follows by Propositions \ref{MoreEstimatesv}, \ref{MoreEstimatesBisv}, in particular $\partial _t v \in L^2_{\mathrm{loc}}(\R_+;\lty)$. Actually the time derivative $\partial _t v$ belongs to $L^\infty _{\mathrm{loc}}(\R_+;\lty)$. This comes immediately by the regularity of $\ave{D}$. Indeed, by the proofs of Propositions \ref{MoreEstimatesv}, \ref{MoreEstimatesBisv} we know that $\divz (R \ave{D}) \in \liy, R\ave{D}\,^t R \in \liy$ and we obtain
\begin{align*}
\partial _t v & = \divz (\ave{D}\nabla _z v) = \divz(\ave{D}\,^t R \nrz v ) = \divz (R \ave{D}) \cdot \nrz v + R\ave{D}:\partial \nrz v \\
& = \divz (R \ave{D}) \cdot \nrz v + R\ave{D}\,^t R : \nrz \otimes \nrz v \in L^\infty_{\mathrm{loc}}(\R_+;\lty).
\end{align*}
We concentrate now on the strong convergence of $(\ve)_{\eps>0}$ in $L^\infty_{\mathrm{loc}}(\R_+;\lty)$ and $(\nabla _z \ve )_{\eps >0}$ in $L^\infty_{\mathrm{loc}}(\R_+;X_P)$. By the energy balance associated with \eqref{Equ8} we deduce
\begin{equation}
\label{E103} \|\ve (t) \|^2 _{\lty} + 2 \int _0 ^t \avepq{\nabla _z \ve (\tau)\otimes \nabla _z \ve (\tau)}{G(\tau/\eps)D}\;\md \tau = \|\uin\|^2 _{\lty},t \in \R_+.
\end{equation}
Similarly, the energy balance associated with \eqref{Equ14} gives
\begin{equation}
\label{E104} \|v (t) \|^2 _{\lty} + 2 \int _0 ^t \avepq{\nabla _z v (\tau)\otimes \nabla _z v (\tau)}{\ave{D}}\;\md \tau = \|\uin\|^2 _{\lty},t \in \R_+.
\end{equation}
By the first statement in Proposition \ref{StrongLim} we know that 
\begin{equation}
\label{E105}
\int _0 ^t \avepq{\nabla _z v (\tau) \otimes \nabla _z v (\tau)}{\ave{D}}\;\md \tau \leq \liminf _{\eps \searrow 0} \int _0 ^t \avepq{\nabla _z \ve (\tau) \otimes \nabla _z \ve (\tau)}{G(\tau/\eps)D}\;\md \tau.
\end{equation}
Combining \eqref{E103}, \eqref{E104}, \eqref{E105} one gets
\begin{align*}
\frac{1}{2}\limsup _{\eps \searrow 0} \{ \|\ve (t) \|^2 _{\lty} & - \|v(t) \|^2 _{\lty} \} = \limsup _{\eps \searrow 0} \left \{\int _0 ^t \avepq{\nabla _z v (\tau) \otimes \nabla _z v (\tau)}{\ave{D}}\;\md \tau \right.\\
& - \left. \int _0 ^t \avepq{\nabla _z \ve (\tau) \otimes \nabla _z \ve (\tau)}{G(\tau/\eps)D}\;\md \tau\right \} \\
& = \int _0 ^t \avepq{\nabla _z v (\tau) \otimes \nabla _z v (\tau)}{\ave{D}}\;\md \tau  \\
& - \liminf _{\eps \searrow 0} \int _0 ^t \avepq{\nabla _z \ve (\tau) \otimes \nabla _z \ve (\tau)}{G(\tau/\eps)D}\;\md \tau \leq 0
\end{align*}
saying that at any time $t \in \R_+$ we have
\begin{equation}
\label{E106} \limsup _{\eps \searrow 0} \|\ve (t) \|^2 _{\lty} \leq \|v(t)\|^2_{\lty}. 
\end{equation}
Applying Fatou lemma to the family of non negative functions $t \to \|\uin \|^2 _{\lty} - \|\ve (t) \|^2 _{\lty}$ we deduce that 
\[
\int _0 ^T \liminf _{\eps \searrow 0}\{ \|\uin \|^2 _{\lty} - \|\ve (t) \|^2 _{\lty}\}\;\md t \leq \liminf _{\eps \searrow 0}\int _0 ^T\{ \|\uin \|^2 _{\lty} - \|\ve (t) \|^2 _{\lty}\}\;\md t 
\]
or equivalently
\[
\limsup _{\eps \searrow 0}\int _0 ^T\|\ve (t) \|^2 _{\lty} \;\md t \leq \int _0 ^T \limsup _{\eps \searrow 0} \|\ve (t) \|^2 _{\lty}\;\md t.
\]
Therefore, the above inequality, together with the weak convergence of the family  $(\ve)_{\eps >0}$ in $L^2([0,T];\lty)$ toward $v$ and \eqref{E106} imply
\begin{align*}
\limsup _{\eps \searrow 0}\int _0 ^T\|\ve (t) \|^2 _{\lty} \;\md t \leq \int _0 ^T \|v(t)\|^2 _{\lty} \;\md t
\end{align*}
saying that $(\ve)_{\eps>0}$ converges strongly in $L^2([0,T];\lty)$ toward $v$ for any $T \in \R_+$
\[
\lime \int _0 ^T \|\ve (t) - v(t)\|^2 _{\lty}\;\md t = 0.
\]
There is a sequence $(\tepsk)_k$ converging to $0$ such that 
\begin{equation}
\label{E107} \limk \|v ^{\tepsk}(t) - v(t)\|^2 _{\lty} = 0,\;\;\mbox{for a.a. } t \in [0,T].
\end{equation}
As $\partial _t v \in L^2([0,T];\lty)$ and $\supe \|\partial _t \ve \|_{L^2([0,T];\lty)} < +\infty$, it is easily seen that \eqref{E107} holds true for any $t \in [0,T]$, $T \in \R_+$, and thus for any $t \in \R_+$. Actually we have
\[
\lime \|\ve (T) - v(T) \|^2 _{\lty} = 0,\;\;T \in \R_+
\]
which implies, thanks to \eqref{E103}, \eqref{E104}
\begin{align*}
\limsup _{\eps \searrow 0} \int _0 ^T \avepq{\nabla _z \ve (t) \otimes \nabla _z \ve (t)}{G(t/\eps)D}\;\md t & = \frac{1}{2} \|\uin \|^2 _{\lty} - \frac{1}{2}\lime \|\ve (T) \|^2 _{\lty} \\
& = \frac{1}{2} \|\uin \|^2 _{\lty} - \frac{1}{2} \|v (T) \|^2 _{\lty}\\
& = \int _0 ^T \avepq{\nabla _z v (t) \otimes \nabla _z v (t)}{\ave{D}}\;\md t.
\end{align*}
By the third statement of Proposition \ref{StrongLim} we deduce that $(\nabla _z \ve)_{\eps>0}$ converges strongly in $L^2([0,T];X_P)$ toward $\nabla _z v$, for any $T \in \R_+$. Finally, in order to prove the convergence of $(\ve )_{\eps>0}$ in $L^\infty(\R_+;\lty)$ toward $v$ we take the difference between the equations \eqref{Equ8} and \eqref{Equ14}
\[
\partial _t ( \ve - v ) - \divz \{ G(t/\eps)D \nabla _z \ve - \ave{D} \nabla _z v\} = 0,\;\;(t,z) \in \R_+ \times \R^m. 
\]
Writing the energy balance, we obtain for any $t \in \R_+$
\begin{align*}
\frac{1}{2}\|\ve (t) - v(t) \|^2 _{\lty}  &+ \int _0 ^t \avepq{[\nabla _z \ve (\tau) - \nabla _z v(\tau)] \otimes \nabla _z \ve (\tau)}{G(\tau/\eps)D}\;\md \tau \\
& - \int _0 ^t \avepq{[\nabla _z \ve (\tau) - \nabla _z v(\tau)] \otimes \nabla _z v (\tau)}{\ave{D}}\;\md \tau = 0.
\end{align*}
As in the proof of Proposition \ref{WeakLim} we have
\begin{align*}
\left | \avepq{[\nabla _z \ve (\tau) - \nabla _z v(\tau)] \otimes \nabla _z \ve (\tau)}{G(\tau/\eps)D}  \right | & \leq |D|_{H_Q^\infty} |\nabla _z \ve (\tau) - \nabla _z v(\tau) |_P |\nabla _z \ve (\tau) |_P
\end{align*}
and
\begin{align*}
\left | \avepq{[\nabla _z \ve (\tau) - \nabla _z v(\tau)] \otimes \nabla _z v (\tau)}{\ave{D}}  \right | & \leq |D|_{H_Q^\infty} |\nabla _z \ve (\tau) - \nabla _z v(\tau) |_P |\nabla _z v (\tau) |_P
\end{align*}
and we deduce that for any $t \in [0,T]$ we have
\[
\|(\ve - v)(t) \|^2 _{L^2}  \leq 2|D|_{H_Q^\infty} \|\nabla _z \ve  - \nabla _z v \|_{L^2([0,T];X_P)} ( \|\nabla _z \ve \|_{L^2([0,T];X_P)} + \|\nabla _z v \|_{L^2([0,T];X_P)}).
\]
The strong convergence of $(\ve)_{\eps >0}$ in $L^\infty([0,T];\lty)$ toward $v$ comes by the strong convergence of $(\nabla _z \ve )_{\eps >0}$ in $L^2([0,T];X_P)$ toward $\nabla _z v$, when $\eps \searrow 0$. 
\end{proof}
\begin{remark}
The strong convergence of $(\ve)_{\eps >0}$ in $L^\infty_{\mathrm{loc}}(\R_+; \lty)$, when $\eps \searrow 0$, holds true for initial conditions $\uin \in \lty$. Indeed, for any $\uin \in \lty$, $T \in \R_+, \delta >0$, let us consider $\uin_{\delta}\in H^{2}_{R}$ such that $\|\uin-\uin_{\delta}\|_{\lty} \leq \delta/2$. We denote by $\ve _\delta$ (resp. $v_\delta$) the variational solution of \eqref{Equ8} (resp. \eqref{Equ14}) with the initial condition $\uin _\delta$. Thanks to the energy balance we obtain easily 
\begin{align*}
\|v^{\eps}-v\|_{L^{\infty}([0,T];\lty)}& \leq \|v^{\eps}-v^{\eps}_{\delta}\|_{L^{\infty}([0,T];\lty)}+\|v^{\eps}_{\delta}-v_{\delta}\|_{L^{\infty}([0,T];\lty)}\\
& +\|v_{\delta}-v\|_{L^{\infty}([0,T];\lty)}\\
& \leq 2 \|\uin - \uin _\delta \|_{\lty} + \|\ve _\delta - v_\delta\|_{L^{\infty}([0,T];\lty)}.
\end{align*}
By Theorem \ref{MainRes1} we know that 
\[
\lime \|\ve_\delta - v_\delta \|_{L^\infty([0,T];\lty)} = 0
\]
and therefore 
\[
\limsup _{\eps \searrow 0} \|v^{\eps}-v\|_{L^{\infty}([0,T];\lty)}\leq \delta,\;\;\delta >0
\]
saying that $\lime \|v^{\eps}-v\|_{L^{\infty}([0,T];\lty)} = 0$, for any $T \in \R_+$. 
\end{remark}
\begin{remark}
\label{ChainRuleDiff}
The computations in the proof of Theorem \ref{MainRes1} show that for any smooth matrix field $C$ and any locally integrable function $v = v(z)$ we have
\[
\left ( \divz (G(s) C \,\nabla _z v)\right )_{-s} = \divy \{ C \nabla _y ( v_{-s}) \}\;\;\mbox{ in } \dpri{}.
\]
\end{remark}
The above considerations show that $\ve = v + o(1)$ in $L^\infty_{\mathrm{loc}}(\R_+;\lty)$, when $\eps \searrow 0$. As suggested by \eqref{Equ10}, we expect a convergence rate in $\calO (\eps)$. This can be achieved assuming that the limit solution $v$ is smooth enough and that there is a smooth matrix field $C$ such that 
\begin{equation}
\label{CloseR}
D = \ave{D} + L(C).
\end{equation}
The existence of the matrix field $C$ is essential when constructing the corrector term $u^1$, see \eqref{EquCcorr}. Notice that Proposition \ref{PropOpeL} guarantees that $D - \ave{D} \in \overline{\ran L}$, and thus \eqref{CloseR} holds true if the range of $L$ is closed. Moreover, we will assume without loss of generality that $C \in (\ker L )^\perp$, which implies also that $^t C \in (\ker L )^\perp$. As $D$ is symmetric, so is $\ave{D}$, and thus $L(C- \,^t C) = 0$. Finally $
C- \,^t C\in \ker L \cap (\ker L )^\perp$ = \{0\}, saying that $C$ is symmetric. 
\begin{proof} (of Theorem \ref{MainRes2})\\
We introduce the functions $\tue (t,y) = v(t, \ymty), (t,y) \in \R_+ \times \R^m, \eps >0$. As in the proof of Theorem \ref{MainRes1} we check that $\tue$ is the variational solution of the problem
\[
\left \{ 
\begin{array}{lll}
\partial _t \tue - \divy \{ G(t/\eps)\ave{D}\nabla _y \tue\} + \frac{1}{\eps} b(y) \cdny \tue = 0,& (t,y) \in \R_+ \times \R^m\\
\tue (0,y) = \uin (y),& y \in \R^m.
\end{array}
\right.
\]
For doing that, pick a smooth compactly supported test function $\Phi (t,y)$ and appeal to the weak formulation of $v$, with the test function $\psie (t,z) = \Phi (t, \ytz), (t,z ) \in \R_+ \times \R^m$. By construction, the average matrix field $\ave{D}$ belongs to $\ker L$ implying that $G(t/\eps)\ave{D} = \ave{D}$. Therefore the functions $(\tue)_{\eps>0}$ solve the problems 
\begin{equation}
\label{E111}
\left \{ 
\begin{array}{lll}
\partial _t \tue - \divy \{\ave{D}\nabla _y \tue\} + \frac{1}{\eps} b(y) \cdny \tue = 0,& (t,y) \in \R_+ \times \R^m\\
\tue (0,y) = \uin (y),& y \in \R^m.
\end{array}
\right.
\end{equation}
Recall that the functions $(\ue)_{\eps>0}$ satisfy
\begin{equation}
\label{E113}
\left \{ 
\begin{array}{lll}
\partial _t \ue - \divy \{D\nabla _y \ue\} + \frac{1}{\eps} b(y) \cdny \ue = 0,& (t,y) \in \R_+ \times \R^m\\
\ue (0,y) = \uin (y),& y \in \R^m.
\end{array}
\right.
\end{equation}
Notice that both families $(\tue)_{\eps>0}, (\ue)_{\eps>0}$ verify the same initial condition. The key point for obtaining a convergence rate is to introduce a corrector term. We consider the function
\begin{align}
\label{EquCcorr}
u^1(t,s,y) & = - \divz (C \nabla _z v(t))(\ymsy) + \divy \{C(y)\nabla _y v (t, \ymsy)\} \\
& = - \tau (-s) \divz (C \nabla _z v(t)) + \divy \{C(y) \nabla _y [\tau (-s) v(t)]\},\;\;(t,s,y) \in \R_+\times \R \times \R^m \nonumber 
\end{align}
where we use the notation $\tau (s) f = f \circ Y(s;\cdot)$ for any function $f$. By Remark \ref{ChainRuleDiff} we have
\begin{equation}
\label{E119}
u^1 (t,s,\ysz) = \divz (G(s)C\nabla _z v(t)) - \divz (C \nabla _z v(t))
\end{equation}
and taking the derivative with respect to $s$ (here $L$ is the infinitesimal generator of the group $G$) leads to
\begin{align*}
\partial _s u^1 (t, s, \ysz) + b(\ysz) \cdny u^1 (t, s, \ysz) & = \divz \left \{\frac{\md }{\md s}G(s)C \,\nabla _z v(t)  \right \}\\
& = \divz \{ G(s)L(C) \nabla _z v(t)\}\\
& = \{\divy[L(C) \nabla _y \tau (-s) v(t)] \} (\ysz).
\end{align*}
Notice that for the last equality we have used one more time Remark \ref{ChainRuleDiff}. Therefore the corrector $u^1$ verifies
\begin{equation}
\label{E117}
(\partial _s  + b(y) \cdny )u^1 (t,s,y) - \divy \{ L(C) \nabla _y v(t,Y(-s;\cdot))\}(y) = 0,\;\;(t,s,y) \in \R_+ \times \R \times \R^m
\end{equation}
and by definition $u^1(t,0,y) = 0,\;(t,y) \in \R _+ \times \R^m$. The equation  \eqref{E117} is exactly the equality coming out at the leading order when plugging the Ansatz $\ue (t,y) = v(t, \ymty) + \eps u^1 (t, t/\eps, y) +...$ into \eqref{E113}. Indeed, the above Ansatz also writes
\[
\ue (t,\ytz) = v(t,z) + \eps u^1 (t,t/\eps, \ytz) + ...
\]
and by observing that 
\begin{align*}
\frac{\md }{\md t}\ue (t, \ytz) & = \partial _t \ue (t, \ytz) + \frac{1}{\eps} b(\ytz) \cdny \ue (t, \ytz) \\
& = [ \divy (D \nabla _y \ue (t)   ) ](\ytz) = \divz [ G(t/\eps) D \nabla _z \ue (t,\ytz)]
\end{align*}
we obtain
\begin{align}
\label{E118}
\partial _t v(t,z) & + \eps \partial _t u^1 (t,t/\eps, \ytz) + \partial _s u^1 (t,t/\eps, \ytz) \\
& + b(\ytz) \cdny u^1 (t,t/\eps, \ytz) + ...  = \divz(G(t/\eps)D \nabla _z v ) \nonumber \\
& + \divz[G(t/\eps)D \nabla _z ( \eps u^1 (t,t/\eps, \ytz))] + ...\;.\nonumber
\end{align}
Taking into account that $\partial _t v = \divz(\ave{D}\nabla _z v)$, we deduce from \eqref{E118}, thanks to Remark \ref{ChainRuleDiff}, that 
\begin{align*}
(\partial _s  + b(y) \cdny )u^1 (t,s,y) & = \tau (-s) \divz [ G(s)D \nabla _z v(t) ] - \tau (-s) \divz [\ave{D}\nabla _z v(t)]\\
& = \divy [(D - \ave{D})\nabla _y \tau (-s)v(t)] \\
& = \divy [L(C) \nabla _y \tau (-s) v(t)]
\end{align*}
which corresponds to \eqref{E117}. In particular, for $s = t/\eps$, one gets
\[
(\partial _s  + b(y) \cdny )u^1 (t,t/\eps,y) - \divy(L(C)\nabla _y \tue (t))(y) = 0,\;(t,y) \in \R_+ \times \R^m,\eps >0
\]
and we obtain the following equation for $\tueo : = u^1 (t, t/\eps, y)$
\begin{align}
\label{E115}
\partial _t (\eps \tueo )(t,y) - \divy (L(C)\nabla _y \tue (t)) + \frac{1}{\eps} b(y) \cdny (\eps \tueo)(t,y) = \eps \partial _t u^1 (t,t/\eps,y).
\end{align}
Taking the sum between the equation in \eqref{E111} and \eqref{E115} yields
\begin{equation}
\label{E116}
\partial _t (\tue + \eps \tueo ) - \divy [(\ave{D} + L(C)) \nabla _y \tue ] + \frac{1}{\eps}b(y) \cdny (\tue + \eps \tueo ) = \eps \partial _t u^1 (t,t/\eps, y)
\end{equation}
which also writes, thanks to \eqref{CloseR}
\[
\partial _t (\tue + \eps \tueo ) - \divy [D \nabla _y (\tue + \eps \tueo)] + \frac{1}{\eps}b(y) \cdny (\tue + \eps \tueo ) = \eps[ \partial _t u^1  -  \divy (D\nabla _y u^1)](t,t/\eps, y).
\] 
Combining \eqref{E113} and \eqref{E116}, it is easily seen that 
\begin{align*}
\partial _t ( \ue - \tue - \eps \tueo) - \divy [D\nabla _y ( \ue - \tue - \eps \tueo)] & + \frac{1}{\eps} b(y)\cdny ( \ue - \tue - \eps \tueo) \\
& = -\eps[ \partial _t u^1  -  \divy (D\nabla _y u^1)](t,t/\eps, y).
\end{align*}
Using the energy balance together with the hypothesis $Q^{1/2}DQ^{1/2} \geq \alpha I_m$ we obtain
\begin{align*}
\frac{1}{2}\frac{\md }{\md t} \|\ue - \tue - \eps \tueo\|^2 _{\lty} & + \alpha |\nabla _y (\ue - \tue - \eps \tueo)|^2_P  \leq \eps 
 \|\ue - \tue - \eps \tueo\|_{\lty}\\
& \times \|\partial _t u^1 (t,t/\eps,\cdot) - \divy (D \nabla _y u^1 (t,t/\eps,\cdot)\|_{\lty},\;\;t \in \R_+.
\end{align*}
Notice that $(\ue - \tue - \eps \tueo)|_{t=0} = \uin - \uin - 0= 0$ and therefore, after integration with respect to $t \in [0,T]$, one gets
\[
\|\ue - \tue - \eps \tueo\|_{L^\infty([0,T];\lty)} \leq \eps \int _0 ^T \|\partial _t u^1 (t,t/\eps,\cdot) - \divy (D \nabla _y u^1 (t,t/\eps,\cdot)\|_{\lty}\;\md t
\]
and
\begin{align*}
\alpha \int _0 ^T |\nabla _y (\ue - \tue - \eps \tueo)|_P ^2 \;\md t & \leq \eps \|\ue - \tue - \eps \tueo\|_{L^\infty([0,T];\lty)}\\
& \times \int _0 ^T \|\partial _t u^1 (t,t/\eps,\cdot) - \divy (D \nabla _y u^1 (t,t/\eps,\cdot)\|_{\lty}\;\md t\\
& \leq \eps ^2 \left ( \int _0 ^T \|\partial _t u^1 (t,t/\eps,\cdot) - \divy (D \nabla _y u^1 (t,t/\eps,\cdot)\|_{\lty}\;\md t\right ) ^2.
\end{align*}
We are done if the corrector $u^1(t,s,y)$ satisfies uniform estimates with respect to the fast variable $s$
\[
u^1 \in L^\infty([0,T];L^\infty(\R_s;\lty)),\;\;\partial _t u^1 \in L^1([0,T];L^\infty(\R_s;\lty))
\]
\[
\divy(D\nabla _y u^1) \in L^1([0,T];L^\infty(\R_s;\lty)),\;\;\nabla _y u^1 \in L^2([0,T];L^\infty(\R_s;X_P)).
\]
Let us estimate the $\lty$ norm of $u^1$, uniformly with respect to $(t,s) \in [0,T]\times \R$. Thanks to \eqref{E119} we have
\begin{align*}
\|u^1 (t,s,\cdot) \|_{\lty}& \leq \|\divz (G(s)C \nabla _z v(t)) - \divz(C \nabla _z v(t))\|_{\lty} \\
& \leq 2 \sup _{s\in \R} \|\divz (G(s)C \nabla _z v(t))\|_{\lty}.
\end{align*}
For any $s\in \R$ we can write, using the formula $\divz(X\xi) = \divz \,^t X \cdot \xi + \,^t X : \partial _z \xi$, for any smooth matrix field $X$ and vector field $\xi$
\begin{align}
\label{E120}
\divz(G(s)C \nabla _z v(t)) & = \divz (G(s)C \,^t R \nrz v(t)) \\
& = \divz(RG(s)C) \cdot \nrz v(t) + RG(s)C\,^t R: \partial _z \nrz v(t)R^{-1} \nonumber \\
& = \divz(RG(s)C) \cdot \nrz v(t) + RG(s)C\,^t R : \nrz \otimes \nrz v(t). \nonumber
\end{align}
We claim that $\divz(RG(s)C) = \tau (s)\divy (RC)$. Indeed, for any smooth compactly supported vector field $\Phi = \Phi (y)$ we have, thanks to \eqref{E43}
\begin{align}
\label{DivRG}
\intz{\divz (RG(s)C) \cdot \Phi (\ysz)} & = - \intz{RG(s)C : \partial _z \{\Phi (\ysz)\}}\\
& = - \intz{RG(s)C\,^t R:(\partial _y \Phi)(\ysz) \partial \ysz R^{-1}}\nonumber \\
& = - \intz{(RC\,^tR)(\ysz): (\partial _y \Phi R^{-1})(\ysz)} \nonumber\\
& = - \inty{RC\,^tR: \partial _y \Phi R^{-1}}\nonumber\\
& = - \inty{RC:\partial _y \Phi} \nonumber\\
& = \inty{\divy(RC)\cdot \Phi (y)} \nonumber\\
& = \intz{\tau (s) [\divy (RC)] \cdot \Phi (\ysz)}.\nonumber
\end{align}
Coming back to \eqref{E120} we obtain
\begin{equation}
\label{E122}
\divz (G(s)C\nabla _z v(t))  = \tau (s) [\divy (RC)] \cdot \nrz v(t) + \tau (s) (RC\,^tR): \nrz \otimes \nrz v(t)
\end{equation}
and therefore
\[
\|\divz (G(s)C\nabla _z v(t))\|_{\lty} \leq \|\divy (RC)\|_{\liy} \|\nrz v(t)\|_{\lty} + |C|_{H_Q ^\infty} \|\nrz \otimes \nrz v(t)\|_{\lty}
\]
saying that 
\begin{align*}
\|u^1\|_{L^\infty([0,T];L^\infty(\R_s;\lty))}& \leq 2 \|\divy (RC)\|_{\liy} \|\nrz v\|_{L^\infty([0,T];\lty)} \\
& + 2 |C|_{H_Q ^\infty} \|\nrz \otimes \nrz v(t)\|_{L^\infty([0,T];\lty)}.
\end{align*}
Similarly, taking the derivative of \eqref{E119} with respect to $t$ yields
\begin{align*}
\|\partial _t u^1\|_{L^1([0,T];L^\infty(\R_s;\lty))}& \leq 2 \|\divy (RC)\|_{\liy} \|\nrz \partial _t v\|_{L^1([0,T];\lty)} \\
& + 2 |C|_{H_Q ^\infty} \|\nrz \otimes \nrz \partial _t v(t)\|_{L^1([0,T];\lty)}.
\end{align*}
It remains to estimate the space derivatives of $u^1$. The key point is that $\nabla ^R $ commutes with $\tau (s)$, i.e.
\[
\nrz (\tau(s) f) = \nrz \{f(\ys)\}= (\nry f)(\ys) = \tau (s) (\nry f)
\]
for any smooth function $f = f(y)$. Indeed, for any $i \in \{1,...,m\}$ we have
\begin{align*}
b_i \cdnz (\tau (s) f )(z) & = \lim_{h\to 0} \frac{f(Y(s;Y_i(h;z))) - f(\ysz)}{h} \\
& =  \lim_{h\to 0} \frac{f(Y_i(h;Y(s;z))) - f(\ysz)}{h}\\
& = b_i (\ysz) \cdot (\nabla _y f)(\ysz) = \tau (s) (b_i \cdny f)(z).
\end{align*}
Applying the operator $\nabla ^R$ in \eqref{E119} and using \eqref{E120}, \eqref{E122} lead to
\begin{align}
\label{E123}
(\nry u^1(t,s,\cdot))(\ys) & = \nrz u^1 (t,s,\ys) = \nrz [\divz(G(s)C\nabla _z v(t)) - \divz(C\nabla _z v(t))]\nonumber \\
& = \nrz [\tau (s) (\divy (RC)) \cdot \nrz v(t) + \tau (s) (RC\,^tR):\nrz \otimes \nrz v(t)]\nonumber \\
& - \nrz [ \divz (RC) \cdot \nrz v(t) +  (RC\,^tR):\nrz \otimes \nrz v(t)].
\end{align}
Appealing one more time to the commutation between $\tau (s) $ and $\nabla ^R$ we deduce that for any $k \in \{1,...,m\}$
\begin{align}
\label{E124}
& b_k \cdnz[\tau (s) (\divy (RC)) \cdot \nrz v(t) + \tau (s) (RC\,^tR):\nrz \otimes \nrz v(t)] \\
& = \tau (s) (b_k \cdny \divy(RC)) \cdot \nrz v(t) + \tau (s) \divy(RC) \cdot (b_k \cdnz \nrz v(t)) \nonumber \\
& + \tau (s) (b_k \cdny  (RC\,^tR)) : \nrz \otimes \nrz v(t) + \tau (s)(RC\,^tR): b_k \cdnz (\nrz \otimes \nrz v(t)).\nonumber
\end{align}
Therefore there is a constant $K$ depending on $\|\divy (RC)\|_{\liy} + \|RC\,^tR \|_{\liy} + \sum_{k = 1} ^m \|b_k \cdny \divy(RC)\|_{\liy} + \sum_{k = 1} ^m\|b_k \cdny (RC\,^tR)\|_{\liy} $ such that 
\begin{align*}
\|\nry u^1 (t,\cdot,\cdot)\|_{L^\infty(\R_s;\lty)} & \leq K\{\|\nrz v(t)\|_{\lty} + \|\nrz \otimes \nrz v(t)\|_{\lty} \\
& + \|\nrz \otimes \nrz \otimes \nrz v(t)\|_{\lty}\}.
\end{align*}
We deduce that 
\begin{align*}
\|\nabla _y u^1 \|_{L^2([0,T];L^\infty(\R_s;X_P))} & \leq K\{\|\nrz v\|_{L^2([0,T];\lty)} + \|\nrz \otimes \nrz v\|_{L^2([0,T];\lty)} \\
& + \|\nrz \otimes \nrz \otimes \nrz v\|_{L^2([0,T];\lty)}\}.
\end{align*}
For the second space derivatives of $u^1$, we write as before
\[
\divy (D \nabla _y u^1)  = \divy (D \,^t R \nry u^1) = \divy (RD) \cdot \nry u^1 + RD\,^t R : \nry \otimes \nry u^1.
\]
Notice that $\divy (RD) \cdot \nry u^1$ belongs to $L^1([0,T];L^\infty(\R_s;\lty))$ since by hypotheses $\divy (RD) \in \liy$ and we already know that $\nabla _y u^1 \in L^2([0,T];L^\infty(\R_s;X_P))$. As the matrix field $RD\,^t R$ belongs to $\liy$, it remains to check that $\nry \otimes \nry u^1$ belongs to $ L^1([0,T];L^\infty(\R_s;\lty))$. For doing that, we apply one more time the operator $\nrz$ in \eqref{E123}, or equivalently the operator $b_l \cdnz$ in \eqref{E124}. Using again the commutation between $\tau (s)$ and $b_l \cdnz$ we obtain
\begin{align*}
& b_l \cdnz \{ b_k \cdnz [\tau (s) (\divy (RC)) \cdot \nrz v(t) + \tau (s) (RC\,^tR):\nrz \otimes \nrz v(t)]\} \\
& = [b_l \cdny(b_k \cdny (\divy(RC)))]_s \cdot \nrz v(t) +  [b_k \cdny (\divy(RC))]_s \cdot [b_l \cdnz(\nrz v(t))]\\
& + [b_l \cdny(\divy(RC))]_s \cdot (b_k \cdnz (\nrz v(t))) + (\divy(RC))_s \cdot \{b_l \cdnz [b_k \cdnz(\nrz v(t))]\}\\
& +  [b_l \cdny (b_k \cdny (RC\,^tR))]_s : \nrz \otimes \nrz v(t) +  [b_k \cdny (RC\,^tR)]_s:b_l \cdnz (\nrz \otimes \nrz v(t))\\
& +  [b_l \cdny (RC\,^tR)]_s : b_k \cdnz (\nrz \otimes \nrz v(t)) +  (RC\,^tR)_s : b_k \cdnz (b_l \cdnz (\nrz \otimes \nrz v(t)))
\end{align*}
which belongs to $L^1([0,T];L^\infty(\R_s;\lty))$, thanks to the hypotheses on the matrix field $C$ and the solution $v$. 
\end{proof}

\appendix
\section{Proofs of Propositions \ref{MoreEstimatesBis}, \ref{MoreEstimatesv}}
\label{A}

\begin{proof}(of Proposition \ref{MoreEstimatesBis})\\
For any $i,j,k\in \{1,...,m\}$ we introduce the notations $\ue _i = b_i \cdny \ue, \ve _i = b_i \cdnz \ve, \ue _{ij} = b_j \cdny ( b_i \cdny \ue), \ve_{ij} = b_j \cdnz (b_i \cdnz \ve), \ue_{ijk} = b_k \cdny (b_j \cdny (b_i \cdny \ue))$ and $\ve _{ijk} = b_k \cdnz (b_j \cdnz (b_i \cdnz \ve))$. With these notations, the equation \eqref{E48} becomes
\[
\partial _t \ue _i - \divy (D(y) \nabla _y \ue _i) + \frac{1}{\eps} b \cdny \ue _i= \divy ([b_i,D]\,^t R \nry \ue ) + D\,^t R \nry \ue \cdny \divy b_i.
\]
Taking now the directional derivative $b_j \cdny $, yields
\begin{align}
\label{E59}
\partial _t \ue _{ij} - \divy (D(y) \nabla _y \ue _{ij}) & + \frac{1}{\eps} b \cdny \ue _{ij}  = \divy ([b_j,D]\,^t R \nry \ue_i ) + R D\;^t R \nry \ue _i \cdot \nry \divy b_j \nonumber \\
& + b_j \cdny \divy ([b_i,D]\nabla _y \ue ) + b_j \cdny (D \,^t R \nry \ue  \cdny \divy b_i).
\end{align}
Thanks to the commutation formula \eqref{ComFor}, we have
\begin{align}
\label{E60} 
& b_j \cdny \divy ([b_i,D]\nabla _y \ue)  = [b_j \cdny, \divy ([b_i,D]\nabla _y )]\ue + \divy ([b_i,D]\nabla _y (b_j \cdny \ue))\nonumber \\
& = \divy ([b_j, [b_i,D]]\nabla _y \ue ) + [b_i,D]\nabla _y \ue \cdot \nabla _y \divy b_j + \divy ([b_i,D]\nabla _y (b_j \cdny \ue)) \nonumber \\
& = \divy ([b_j, [b_i,D]]\,^t R \nry \ue ) + R[b_i,D]\,^t R \nry \ue \cdot \nry \divy b_j + \divy ([b_i,D]\,^t R\nry \ue _j).
\end{align}
Combining \eqref{E59}, \eqref{E60} we obtain
\begin{align*}
\partial _t \ue _{ij} - \divy (D(y) \nabla _y \ue _{ij} ) & + \frac{1}{\eps} b \cdny \ue _{ij}  = \divy ([b_j,D]\,^t R \nry \ue _i ) + R D \,^t R \nry \ue _i \cdot \nry \divy b_j \\
& + \divy ([b_j, [b_i,D]]\,^t R \nry \ue ) + R[b_i,D]\,^t R \nry \ue \cdot \nry \divy b_j\\
& + \divy([b_i,D]\,^t R \nry \ue _j) + b_j \cdny (D\,^t R \nry \ue \cdny \divy b_i).
\end{align*}
Multiplying by $\ue _{ij}$ and integrating on $\R^m$ lead to
\begin{align}
\label{E61}
\frac{1}{2}\frac{\md }{\md t} \inty{(\ue _{ij})^2} & + \inty{D\nabla _y \ue _{ij} \cdot \nabla _y \ue _{ij}} = - \inty{R[b_j,D]\,^t R \nry \ue _i \cdot \nry \ue _{ij} } \nonumber \\
& + \inty{RD\,^t R\nry \ue _i \cdot \nry \divy b_j \, \ue _{ij}} + \inty{\divy ([b_j,[b_i,D]]\,^t R \nry \ue ) \ue _{ij}}\nonumber \\
& + \inty{R[b_i,D]\,^t R\nry \ue \cdot \nry (\divy b_j) \ue _{ij}} + \inty{\ue _{ij} \divy ( [b_i,D]\,^t R \nry \ue _j)}\nonumber \\
& + \inty{\ue _{ij} b_j \cdny (D \,^t R \nry \ue \cdny (\divy b_i))}\nonumber \\
& =: K^1 _{ij} + K^2 _{ij} +K^3 _{ij} +K^4 _{ij} +K^5 _{ij} +K^6 _{ij}.
\end{align}
By hypothesis \eqref{E46} we have, cf. \eqref{EquCoerH1R}
\begin{align}
\label{E62} 
D \nabla _y \ue _{ij} \cdot \nabla _y \ue _{ij} & 
\geq \alpha |P^{1/2} \nabla _y \ue _{ij}|^2 = \alpha |\nry \ue _{ij}|^2.
\end{align}
Exactly as before we obtain
\[
\nry (\ue _{ij}) = b_j \cdny (\nry \ue _i) - \mathcal{A}_j \nry \ue _i
\]
which allows us to replace $K^1 _{ij}$ by 
\begin{align*}
K^1 _{ij} & = \inty{R[b_j,D]\,^t R \nry \ue _i (t) \cdot \mathcal{A}_j \nry \ue _i (t)} \\
& + \frac{1}{2} \inty{b_j \cdny (R[b_j,D]\,^t R):\nry \ue _i (t) \otimes \nry \ue _i (t) } \\
& + \frac{1}{2}\inty{(\divy b_j) R [b_j,D]\,^t R : \nry \ue_i (t) \otimes \nry \ue_i (t)}.
\end{align*}
Thanks to our hypotheses, there is a constant $C_3$ (not depending on $\eps$ or $t$) such that
\[
\sum _{j = 1} ^m \sum _{i = 1} ^m |K_{ij} ^1| \leq C_3 \|\nry \otimes \nry \ue (t)\|^2 _{\lty},\;\;t \in \R_+, \;\;\eps >0.
\]
Obviously, there is a constant $C_4$  (not depending on $\eps$ or $t$) such that
\[
\sum _{j = 1} ^m \sum _{i = 1} ^m |K_{ij} ^2| \leq C_4 \|\nry \otimes \nry \ue (t)\|^2 _{\lty},\;\;t \in \R_+, \;\;\eps >0.
\]
We consider now the term $K^3_{ij}$, which writes
\begin{align*}
K^3 _{ij} & = \inty{\ue _{ij} (t) \divy (R[b_j,[b_i,D]])\cdot \nry \ue (t) } \\
& + \inty{\ue _{ij} (t) R[b_j,[b_i,D]]\,^t R :\nry \otimes \nry \ue (t) }.
\end{align*}
It is easily seen that there is a constant $C_5$ (not depending on $\eps$ or $t$) such that
\begin{align*}
\sum _{j = 1} ^m \sum _{i = 1} ^m ( |K_{ij} ^3| + |K_{ij} ^4| + |K_{ij} ^6|)& \leq C_5 \left ( \|\nry \otimes \nry \ue (t)\| _{\lty} \|\nry \ue (t)\|_{\lty} \right.\\
& \left.+\|\nry \otimes \nry \ue (t)\|^2 _{\lty} \right ) ,\;\;t \in \R_+, \;\;\eps >0.
\end{align*}
It remains to estimate the term $K^5_{ij}$. For any $i,j \in \{1,...,m\}$ we have
\[
\ue _{ij} (t) = \ue _{ji}(t) - \sum _{k = 1}^m \alpha _{ij} ^k \ue _k (t)
\]
and therefore $K^5 _{ij}$ writes
\begin{align*}
K^5 _{ij} & = - \sum _{k = 1} ^m \inty{\alpha _{ij} ^k \ue _k (t) \divy ([b_i,D]\;^t R \nry \ue _j (t))} \\
& + \inty{\ue _{ji} (t) \divy ( [b_i,D]\,^t R \nry \ue _j (t))} \\
& = \sum _{k = 1} ^m \inty{\nry (\alpha _{ij} ^k \ue _k (t))\cdot R[b_i,D]\,^t R \nry \ue _j (t)}\\
& - \inty{\nry (b_i \cdny \ue _j (t)) \cdot R [b_i,D]\,^t R \nry \ue _j (t)} = : K^7_{ij} + K^8_{ij}.
\end{align*}
Clearly, there is a constant $C_6$ (not depending on $\eps$ or $t$) such that
\begin{align*}
\sum _{j = 1} ^m \sum _{i = 1} ^m  |K_{ij} ^7| & \leq C_6 \left ( \|\nry \otimes \nry \ue (t)\| _{\lty} \|\nry \ue (t)\|_{\lty} \right.\\
& \left.+\|\nry \otimes \nry \ue (t)\|^2 _{\lty} \right ) ,\;\;t \in \R_+, \;\;\eps >0.
\end{align*}
For the last term $K^8 _{ij}$ we use \eqref{E54} and we get as before
\begin{align*}
K^8 _{ij} & = \inty{\mathcal{A}_i \nry \ue _j (t) \cdot R[b_i,D]\,^t R \nry \ue _j (t)} \\
& - \inty{b _i \cdny ( \nry \ue _j (t)) \cdot R[b_i, D]\,^t R \nry \ue _j (t)}\\
& = \inty{\mathcal{A}_i \nry \ue _j (t) \cdot R[b_i, D]\,^t R \nry \ue _j (t)}\\
& + \frac{1}{2}\inty{(\divy b_i) R[b_i, D]\,^t R: \nry \ue _j (t) \otimes \nry \ue _j (t)}\\
& + \frac{1}{2}\inty{b_i \cdny (R[b_i, D]\,^t R) : \nry \ue _j (t) \otimes \nry \ue _j (t)}
\end{align*}
implying that there is a constant $C_7$ (not depending on $\eps$ or $t$) such that
\begin{align*}
\sum _{j = 1} ^m \sum _{i = 1} ^m  |K_{ij} ^8| & \leq C_7 \|\nry \otimes \nry \ue (t)\| _{\lty} ^ 2,\;\;t \in \R_+, \;\;\eps >0.
\end{align*}
Putting together \eqref{E61}, \eqref{E62} and the estimates for all the terms $K_{ij} ^r, i,j \in \{1, ...,m\}, r \in \{1, ...,8\}$ we deduce that 
\begin{align*}
& \frac{1}{2}\frac{\md }{\md t} \|\nry \otimes \nry \ue \|^2 _{\lty} + \alpha \|\nry \otimes \nry \otimes \nry \ue \|^2 _{\lty}\leq C \| \nry \otimes \nry \ue \| _{\lty}  \\
& \times \left (\| \nry \otimes \nry \ue \| _{\lty} + \|  \nry \ue \| _{\lty}     \right ),\;C = \sum _{r = 3} ^7 C_r.
\end{align*}
Applying Gronwall's lemma yields, for some constant $C_T$ depending only on $T$ and the coefficients $\alpha_{ij}^k$, the vector fields $b_i$ and the matrix field $D$
\begin{align*}
\|\nry \otimes \nry \ue \|_{L^\infty([0,T];\lty)}  & = \|\nrz \otimes \nrz \ve \|_{L^\infty([0,T];\lty)} \\
& \leq C_T ( \|\nry \uin \|_{\lty} + \|\nry \otimes \nry \uin \|_{\lty} ),\;\;\eps >0
\end{align*}
and
\begin{align*}
\|\nry \otimes \nry \otimes \nry \ue \|_{L^2([0,T];\lty)} & = \|\nrz \otimes \nrz \otimes \nrz \ve \|_{L^2([0,T];\lty)}  \\
& \leq C_T ( \|\nry \uin \|_{\lty}  + \|\nry \otimes \nry \uin \|_{\lty} ),\eps >0.
\end{align*}
For estimating $\partial _t \nrz \ve$ we take the directional derivative $b_i \cdnz$ in \eqref{DerTimeVeps}. As the vector fields $b_i, b$ are in involution, that is $[b_i,b] = 0$, the flows $Y_i, Y$ are commuting \cite{Arnold78, Arnold89}, and therefore the derivative along $b_i$ commutes with the translation along the flow of $b$ (take the derivative with respect to $h$, at $h = 0$, of the equality $f(Y(s;Y_i(h;\cdot))) = f(Y_i (h; Y(s;\cdot)))$\;). We deduce that
\[
\partial _t ( b_i \cdnz \ve ) (t,z) = [b_i \cdny \divy (D \nabla _y \ue )](\ytz)
\]
which implies 
\begin{align*}
\|\partial _t \nrz \ve (t)\|_{\lty} & = \|\nry \divy (D \nabla _y \ue (t))\|_{\lty} \\
& = \|\nry \divy ( D \,^t R \nry \ue (t))\|_{\lty} \\
& = \|\nry (\divy (RD) \cdot \nry \ue (t)) + \nry ( R D \,^t R : \nry \otimes \nry \ue (t)) \|_{L^2}.
\end{align*}
We claim that for any $i \in \{1, ...,m\}$ we have the equality 
\begin{equation}
\label{EquDivbi}
\divy (RD)_i = \sum _{j =1 } ^m b_j \cdny (RD\,^t R)_{ij} + \sum _{j =1 } ^m (RD\,^t R)_{ij} \divy b_j.
\end{equation}
Indeed, for any $i \in \{1, ...,m\}$ we can write (here $(e_k)_{1\leq k \leq m}$ stands for the canonical basis of $\R^m$)
\begin{align*}
(\divy RD)_i & = \divy (RD \,^t R \,^t R ^{-1})_i = \sum_{k = 1} ^m \partial _{y_k} \sum _{j=1} ^m (RD\,^t R)_{ij} R^{-1}_{kj}\\
& = \sum _{j=1} ^m \sum_{k = 1} ^m \partial _{y_k} [(RD\,^t R)_{ij} (b_j \cdot e_k)] \\
& = \sum _{j=1} ^m \divy [ (RD\,^t R)_{ij} b_j ].
\end{align*}
Thanks to our hypotheses and formula \eqref{EquDivbi}, it is easily seen that there is a constant depending only on the coefficients $\alpha _{ij} ^k$, the vector fields $b_i$ and the matrix field $D$ such that 
\begin{align*}
\|\partial _t \nrz \ve \|_{\lty} & \leq C  (\|\nry \ue (t) \|_{\lty} + \|\nry \otimes \nry \ue (t) \|_{\lty} \\
& + \|\nry \otimes \nry \otimes \nry \ue (t)\|_{\lty}    ).
\end{align*}
Thanks to the uniform estimates satisfied by $\nry \ue, \nry \otimes \nry \ue$ in $L^\infty([0,T];\lty)$, and by $\nry \otimes \nry \otimes \nry \ue $ in $L^2([0,T];\lty)$, we obtain that, for any $T \in \R_+$, there is a constant $\tilde{C}_T$ (depending only on $T$ and $\alpha$, $\alpha _{ij} ^k, b_i, D$) such that 
\[
\|\partial _t \nrz \ve \|_{L^2([0,T];\lty)} \leq \tilde{C}_T \left ( \|\nry \uin \|_{\lty} + \|\nry \otimes \nry \uin \|_{\lty} \right ),\;\;\eps >0.
\]
\end{proof}

\begin{proof}( of Proposition \ref{MoreEstimatesv})\\
We perform exactly the same computations as in the proof of Proposition \ref{MoreEstimates} (it does not matter that \eqref{Equ14} has no the term $\frac{1}{\eps} b \cdnz v$). Nevertheless, we have to check that all the hypotheses on the matrix field $D$ in Proposition \ref{MoreEstimates} are also satisfied by the matrix field $\ave{D}$. By Theorem \ref{AveMatField} we deduce that $^t\ave{D} = \ave{D}$, $\ave{D} \in H_Q \cap H_Q ^\infty$, $|\ave{D}|_Q \leq |D|_Q,|\ave{D}|_{H_Q ^\infty} \leq |D|_{H_Q ^\infty}$, $Q^{1/2} \ave{D} Q^{1/2} \geq \alpha I_m$ and therefore the hypotheses \eqref{E45}, \eqref{E46} corresponding to the matrix field $\ave{D}$ hold true. We also need to show that $R[b_i,\ave{D}]\,^t R \in \liy$, $\divy (R \ave{D}) \in \liy$, $\sum _{i = 1} ^m b_i \cdny (R[b_i, \ave{D}]\,^t R) \in \liy$, provided that the same conditions are satisfied by the matrix field $D$. The key point is that for any $i \in \{1, ...,m\}$, the groups $(G_i (h))_{h \in \R}$, $(G(s))_{s \in \R}$ are commuting, where $G_i (h)$ is defined by
\[
G_i (h) A = \partial Y_i ^{-1}(h;\cdot) A(Y_i(h;\cdot))\;^t \partial Y_i ^{-1} (h;\cdot).
\]
It is easily seen that for any $s, h \in \R$
\begin{align}
\label{E65}
(G(s) \circ G_i (h))A & = G(s) (G_i (h)A) = \partial Y^{-1} (s; \cdot)(G_i (h)A)(\ys)\,^t\partial Y^{-1} (s; \cdot)  \\
& = \partial Y^{-1} (s; \cdot)\partial Y_i^{-1} (h; \ys)A(Y_i(h;\ys)) \,^t \partial Y_i^{-1} (h; \ys)\,^t\partial Y^{-1} (s; \cdot)  \nonumber 
\end{align}
and
\begin{align}
\label{E66}
(G_i(h) \circ G (s))A & = G_i(h) (G (s)A) = \partial Y_i^{-1} (h; \cdot)(G (s)A)(Y_i(h;\cdot))\,^t\partial Y_i^{-1} (h; \cdot)  \\
& = \partial Y_i^{-1} (h; \cdot)\partial Y^{-1} (s; Y_i(h;\cdot))A(Y(s;Y_i(h;\cdot))) \,^t \partial Y^{-1} (s; Y_i(h;\cdot))\,^t\partial Y_i ^{-1} (h; \cdot).  \nonumber 
\end{align}
By the involution between $b$ and $b_i$, we know that 
\begin{equation}
\label{E67} Y_i (h; \ys) = Y(s; Y_i (h;\cdot))
\end{equation}
and by differentiation one gets
\[
\partial Y_i (h; \ys) \partial \ys = \partial Y (s; Y_i (h;\cdot)) \partial Y_i (h;\cdot)
\]
which also writes
\begin{equation}
\label{E68}
\partial Y ^{-1} (s; \cdot) \partial Y_i ^{-1} (h; \ys) = \partial Y_i ^{-1} (h; \cdot) \partial Y ^{-1} (s;Y_i (h;\cdot)).
\end{equation}
Combining \eqref{E65}, \eqref{E66}, \eqref{E67}, \eqref{E68} we obtain the commutation property between the groups $(G_i(h))_{h \in \R}, (G(s))_{s \in \R}$, for any $i \in \{1, ...,m\}$. Notice that the hypothesis $R [b_i, D]\,^t R \in \liy$, or equivalently $[b_i,D] \in H_Q^\infty$, should be understood in $\dpri$, that is,  there is a matrix field, denoted $[b_i,D]$, which belongs to $H_Q^\infty$, such that for any $A \in C^1 _c (\R^m)$ we have
\begin{equation}
\label{E69}
\inty{D:(- b_i \cdny A - (\divy b_i )A - \,^t \partial b_i A - A \partial b_i )} = \inty{[b_i,D]:A}.
\end{equation}
We introduce the operator $L_i (D) = [b_i,D]$ and its formal adjoint
\[
L_i ^\star (A) = - b_i \cdny A - (\divy b_i )A - \,^t \partial b_i A - A \partial b_i,\;A \in C^1 _c (\R^m).
\]
Using these notations, \eqref{E69} becomes
\begin{equation*}
\label{E70} \inty{D:L_i ^\star (A)} = \inty{L_i (D) :A },\;\;A \in C^1 _c (\R^m).
\end{equation*}
Notice that $H_Q^\infty$ is the topological dual of the space 
\[
H_P^1 = \left \{A:\R^m \to \mathcal{M}_m (\R) \mbox{ measurable }\;:\; \inty{(P(y)A(y) : A(y)P(y))^{1/2}} < +\infty \right \}
\]
and thus $[b_i, D] \in H_Q^\infty$ iff there is a constant $C_i$ such that $\inty{D:L_i ^\star (A) } \leq C_i |A|_{H_P^1}$ for any $A \in C^1 _c(\R^m)$. A straightforward computation shows that if $B \in H_Q ^\infty$ is such that $G_i(h)B \in H_Q ^\infty$ for any $h \in \R$ and $h \in \R^\star \to (G_i(h)B - B)/h$ is bounded in $H_Q^\infty$, then $(G_i(h)B - B)/h$ converges toward $L_i (B)$ weakly $\star$ in $H_Q^\infty$, when $h \to 0$. Indeed, for any $A \in C^1 _c (\R^m) \subset H_P^1$, we have
\begin{align*}
& \inty{\frac{G_i(h)B  - B}{h}:A}  = \inty{\frac{\partial Y_i ^{-1} (h;y) B(Y_i (h;y))\,^t \partial Y_i ^{-1} (h;y) - B(y)}{h}: A(y)}\\
& = \inty{\frac{\partial Y_i (-h;Y_i (h;y))B(Y_i(h;y))\,^t \partial Y_i (-h;Y_i (h;y))- B(y)}{h} : A(y)} \\
& = \frac{1}{h} \intz{\left \{\mathrm{det}(\partial Y_i (-h;z)) \partial Y_i (-h;z)B(z) \,^t \partial Y_i (-h;z): A(Y_i (-h;z)) - B(z) :A(z)\right \}}\\
& = \intz{\frac{\mathrm{det}(\partial Y_i (-h;z)) - 1}{h}\partial Y_i (-h;z)B(z) \,^t \partial Y_i (-h;z): A(Y_i (-h;z))} \\
& + \intz{  B(z) :\frac{^t \partial Y_i (-h;z) A(Y_i(-h;z))\partial Y_i (-h;z) - A(z)}{h}}\\
& \underset{h\to 0}{\longrightarrow} - \intz{\divz b_i \;B(z) : A(z)} - \intz{B(z) : ( b_i \cdnz A + \,^t \partial b_i A + A \partial b_i )}\\
& = \intz{B(z):L_i ^\star (A)}. 
\end{align*}
We deduce that any weak $\star$ limit point in $H_Q^\infty$ satisfies
\[
\inty{\mbox{w}\star \lim _{h_k \to 0} \frac{G_i (h_k)B - B}{h_k} :A } = \inty{B:L_i ^\star (A)}
\]
for any $A \in C^1 _c (\R^m)$. Therefore all the family $(G_i (h) B - B)/h$ converges weakly $\star$ in $H_Q ^\infty$, as $h\to 0$, and $L_i(B) = \lim_{h \to 0} \frac{G_i(h)B - B}{h}$, weakly $\star$ in $H_Q^\infty$. 

\noindent
We claim that for any $s\in \R$, we have $L_i (G(s)D) = G(s)(L_i(D))$, that is
\begin{equation}
\label{E71} \inty{G(s)D : L_i ^\star (A)} = \inty{G(s)L_i (D) : A},\;\;A \in C^1 _c (\R^m).
\end{equation}
By density arguments (notice that $B_n \rightharpoonup B$ weakly $\star$ in $H_Q^\infty$, implies $G(s)B_n \rightharpoonup G(s)B$ weakly $\star$ in $H_Q^\infty$, for any $s \in \R$) it is enough to show that $L_i (G(s)B)= G(s)L_i(B)$ for any smooth, compactly supported matrix field $B$. Let us consider a smooth, compactly supported matrix field $B$. Obviously, for any $h \in \R^\star$ we have $G_i (h)B \in H_Q ^\infty$ and $\frac{G_i (h)B-B}{h} \rightharpoonup  L_i(B)$ weakly $\star$ in $H_Q ^\infty$, when $h \to 0$. We deduce that $G_i (h)G(s)B = G(s)G_i (h)B \in H_Q ^\infty$ for any $h \in \R$ and
\[
\frac{G_i (h)G(s)B - G(s)B}{h} = \frac{G(s)G_i (h)B - G(s)B}{h} = G(s)\frac{G_i (h)B - B}{h} \underset{h \to 0}{\rightharpoonup} G(s)L_i (B)
\]
weakly $\star$ in $H_Q^\infty$. By the previous remark, we obtain $L_i(G(s)B) = G(s)L_i(B)$ for any $s \in \R$. Now it is easily seen that $L_i (\ave{D}) = [b_i,\ave{D}] \in H_Q^\infty$, if $L_i (D) = [b_i,D] \in H_Q ^\infty$. Indeed, 
averaging \eqref{E71} with respect to $s$ one gets
\[
\inty{\frac{1}{S} \int _0 ^S G(s)D\;\md s : L_i ^\star (A)} = \inty{\frac{1}{S}\int _0 ^S G(s)L_i (D) \;\md s :A},\;\;A \in C^1 _c (\R^m).
\]
Taking into account that $\frac{1}{S}\int _0 ^S G(s)D \;\md s \to \ave{D}$ in $H_Q$, when $S \to +\infty$, that $PL_i ^\star (A) P \in H_Q$ and noticing that
\[
\left | \frac{1}{S}\int _0 ^S G(s)L_i(D)\;\md s \right|_{H_Q^\infty} \leq \frac{1}{S}\int _0 ^S |G(s)L_i(D)|_{H_Q ^\infty}\;\md s = |L_i (D)|_{H_Q ^\infty}
\]
we deduce that any weak $\star$ limit point in $H_Q^\infty$ of $\left ( \frac{1}{S}\int _0 ^S G(s)L_i (D)\;\md s \right ) _S$ satisfies
\[
\inty{\ave{D}:L_i ^\star (A)} = \inty{\mbox{w}\star \lim _{S \to +\infty}\frac{1}{S}\int _0 ^S G(s)L_i(D)\;\md s : A}
\]
for any $A \in C^1 _c (\R^m) \subset H^1_P$, saying that $L_i (\ave{D})= \lim _{S \to +\infty}\frac{1}{S}\int _0 ^S G(s)L_i(D)\;\md s$, weakly $\star$ in $H_Q^\infty$. In particular $|L_i (\ave{D})|_{H_Q^\infty} \leq |L_i (D) |_{H_Q^\infty}$. 

\noindent
We concentrate now on the hypothesis $\sum _{i = 1} ^m b_i \cdny (R[b_i,D]\,^t R) \in \liy$. For any $s \in \R$ we have
\[
RG(s)[b_i,D]\,^t R = R \partial Y^{-1}(s;\cdot) [b_i,D](\ys) \,^t \partial Y^{-1}(s;\cdot) \,^t R = (R[b_i,D]\,^t R)(\ys)
\]
and since $[b_i,b] = 0$, we obtain
\[
b_i \cdnz (RG(s)[b_i,D]\,^t R ) = b_i \cdnz ((R[b_i,D]\,^t R) \circ \ys ) = \left (b_i \cdny (R[b_i,D]\,^t R ) \right )\circ \ys.
\]
Multiplying by a smooth compactly supported matrix field $A \in C^1 _c (\R^m)$ and averaging  with respect to $s$ one gets 
\begin{align*}
- \intz{ \frac{\divz b_i}{S} \int _0 ^S G(s)L_i (D) \;\md s & : \,^t R AR}  - \intz{\frac{1}{S} \int _0 ^S G(s)L_i (D)\;\md s :\,^t R (b_i \cdnz A)R} \\
& = \intz{\frac{1}{S}\int _0 ^S (b_i \cdny (RL_i (D) \,^t R))(\ysz) \;\md s :A(z)}.
\end{align*}
We use now the weak $\star$ convergence in $H_Q^\infty$
\[
\limS \frac{1}{S} \int _0 ^S G(s)L_i (D) \;\md s = L_i (\ave{D})
\]
and the facts that $^tRAR\in H^1_P$, $\divy b_i \in \liy$ implying that 
\begin{align}
\label{E79}
- \intz{\sum _{i = 1} ^m \divz b_i & \, R L_i (\ave{D}) \,^t R  : A}   - \intz{\sum _{i = 1} ^m RL_i (\ave{D})\,^t R:b_i \cdnz A} \\
& = \intz{\lim_{S\to +\infty} \frac{1}{S}\int _0 ^S \sum _{i = 1} ^m (b_i \cdny (RL_i (D) \,^t R))(\ysz) \;\md s :A(z)}.\nonumber 
\end{align}
For passing to the limit in the last integral we use the weak $\star$ convergence in $\liy{}$, since the family $\left ( \sum _{i = 1} ^m (b_i \cdny (RL_i (D) \,^t R))\circ \ys  \right ) _{s \in \R}$ is bounded in $\liy{}$. We deduce that 
\[
\sum _{i = 1} ^m b_i \cdny (RL_i (\ave{D}) \,^t R) = \mathrm{w}\star \lim_{S\to +\infty} \frac{1}{S}\int _0 ^S \sum _{i = 1} ^m (b_i \cdny (RL_i (D) \,^t R))\circ \ys \;\md s \in \liy{}.
\]
It remains to observe that for any $s \in \R$, $\divz (RG(s)D) = \divy (RD)\circ \ys$ (see \eqref{DivRG} for details), which implies
\[
\divz(R\ave{D}) = \mathrm{w}\star \lim _{S\to +\infty} \frac{1}{S}\int _0 ^S \divy(RD) \circ \ys \;\md s \in \liy
\]
where, as before, the last limit should be understood in the weak $\star \;\liy$ sense.
\end{proof}

\subsection*{Acknowledgments}
This work has been carried out within the framework of the EUROfusion Consortium and has received funding from the Euratom research and training programme 2014-2018 under grant agreement No 633053. The views and opinions expressed herein do not necessarily reflect those of the European Commission.

\end{document}